%% file: revision7.tex
\documentclass[11pt,a4paper]{amsart}

\usepackage{amsfonts,amsmath,amssymb,amsthm}
\usepackage[latin1]{inputenc}

\usepackage{tikz}
\usepackage{pgfplots}
\usetikzlibrary{calc}
\usetikzlibrary{shadows}
\usetikzlibrary{arrows}
\usetikzlibrary{patterns}
\usetikzlibrary{matrix}
\usepackage{float}
\usepackage{todonotes}
\usepackage{mathabx}

\usepackage[a4paper]{geometry}
\newcommand{\Cc}{\mathbb{C}} 

\newcommand{\Rr}{\mathbb{R}}

\renewcommand {\ge}{\geqslant}

\newcommand{\grad}{\mathop{\mathrm{grad}}\nolimits} 
\newcommand{\sgn}{\mathop{\mathrm{sgn}}\nolimits} 

\newcommand{\defi}[1]{\emph{#1}}
\newcommand{\wt}[1]{\widetilde{#1}}

\newcommand{\ol}[1]{\overline{#1}}

\newcommand{\ii}{\mathrm{i}}
\newcommand{\dd}{\mathrm{d}}
\renewcommand{\Re}{\mathop{\mathrm{Re}}\nolimits}
\renewcommand{\Im}{\mathop{\mathrm{Im}}\nolimits}
\newcommand{\midarrow}{\tikz \draw[-triangle 90] (0,0) -- +(.1,0);}

\DeclareMathOperator{\rot}{rot}
\DeclareMathOperator{\writhe}{wr}
\DeclareMathOperator{\area}{area}
\DeclareMathOperator{\ind}{ind}
\DeclareMathOperator{\lk}{lk}

\newcommand{\reid}{\mathbf{\Omega}}
\newcommand{\Oo}{\mathbf{O}}
\newcommand{\Ii}{\mathbf{I}}

\theoremstyle{plain}
\newtheorem{theorem}[equation]{Theorem}    
\newtheorem*{theorem*}{Theorem}
\newtheorem{lemma}[equation]{Lemma}       
\newtheorem{proposition}[equation]{Proposition}      
\newtheorem{proposition*}{Proposition} 
\newtheorem{corollary}[equation]{Corollary}      

\theoremstyle{definition}

\newtheorem{definition}[equation]{Definition}      
\newtheorem{condition}[equation]{Condition}
\theoremstyle{remark}

\newtheorem*{remark*}{Remark}  
\newtheorem{remark}[equation]{Remark}   
\newtheorem{example}[equation]{Example}
\newtheorem{question}[equation]{Question}
\newtheorem*{ack}{Acknowledgements}  

\numberwithin{equation}{section}

\title[Intermediate links]{Intermediate links of plane curves}

\author{Arnaud Bodin}
\email{Arnaud.Bodin@math.univ-lille1.fr}
\address{Laboratoire Paul Painlev\'e, Math\'ematiques, Universit\'e Lille 1, 59655 Villeneuve d'Ascq Cedex, France}

\author{Maciej Borodzik}
\email{mcboro@mimuw.edu.pl}
\address{Institute of Mathematics, University of Warsaw, ul. Banacha 2, 02-097, Warsaw, POLAND}
\address{Institute of Mathematics, Polish Academy of Science, ul \'Sniadeckich 8, Warsaw, POLAND}



\date{\today}

\begin{document}
\begin{abstract}
For a smooth complex curve $\mathcal{C}\subset\Cc^2$ we consider the link $L_r=\mathcal{C}\cap \partial 
B_r$, where $B_r$ denotes an Euclidean ball of radius $r>0$. We prove that
the diagram $D_r$ obtained from $L_r$ by a complex stereographic projection 
satisfies $\chi(\mathcal{C}\cap B_r)=\rot(D_r)-\writhe(D_r)$. As a consequence we show
that if $D_r$ has no negative Seifert circles and $L_r$ is strongly quasipositive
and fibered, then the 
Yamada--Vogel algorithm applied to $D_r$ yields a quasipositive braid.
\end{abstract}
\maketitle


\section{Introduction}

Let $\mathcal{C}\subset\Cc^2$ be a complex algebraic curve. 
Consider $B_r$, the standard 
Euclidean ball with center $0\subset\Cc^2$ and radius $r>0$. Let 
$S_r=\partial B_r$.
Define $L_r=\mathcal{C}\cap S_r$. If the intersection is transverse, then $L_r$ is 
a link in $S_r$, called a \defi{$\Cc$-link} or an \defi{intermediate link} of the plane curve $\mathcal{C}$.
In this paper we will assume that the curve $\mathcal{C}$ is smooth; this is not a strong restriction since any
$\Cc$-link can be obtained as the intersection of a smooth curve and a sphere.

It was conjectured by Rudolph \cite{Ru1983}
and proved by Boileau and Orevkov \cite{BO} that $\Cc$-links are quasipositive. 
A link is called \emph{quasipositive}, if it can be represented by a quasipositive braid. In turn, 
a braid is called \emph{quasipositive} if
it is of the form $\prod w_i\sigma_{k_i}w_i^{-1}$, where $\sigma_j$ are generators in the braid group and $w_i$
are arbitrary words in the braid group.

Quasipositive links  are a subject
of intense research both for their applications in algebraic geometry and for their special properties in link theory;
see \cite{Baa05, Baa14, Hed05, Ru2005} for instance.
Unfortunately the proof that $\Cc$-links are quasipositive links is not constructive: given a $\Cc$-link we lack an algorithm 
translating it into a quasipositive braid. 

Usually a link diagram is thought of as the image of a link under a projection from $\Rr^3$ to $\Rr^2$. For links in $S^3$, in order
to draw a diagram, one usually applies some variant of a stereographic projection. There is a version of a stereographic projection, which seems to
be the most natural when studying $\Cc$-links, namely the complex stereographic projection, see Section~\ref{sec:stereo}. One expects that
the complex stereographic projection of a $\Cc$-link should have some special properties. This idea is exploited in the present article. In particular,
we find a relation between the writhe, the rotation number and the Euler characteristic of a diagram of a $\Cc$-link under
the complex stereographic projection.

In the next step we apply the Yamada--Vogel algorithm (see Section~\ref{sec:vogelalgorithm}) to such a diagram. The Yamada--Vogel algorithm
gives an explicit way of turning a link diagram into a braid. One of the features of this algorithm
is that one can control the Reidemeister moves performed in the algorithm.
We ask, whether the Yamada--Vogel algorithm can turn a diagram of a $\Cc$-link into a quasipositive braid. Our first result gives a sufficient
condition, when this is the case.

\begin{theorem}\label{main1}
Let $L_r$ be a link arising as a transverse intersection of a complex curve $\mathcal{C}\subset\Cc^2$ with a sphere of radius $r$. Let $D_r$
be a diagram of $L_r$ obtained from a complex stereographic projection.
Suppose $L_r$ is fibered strongly quasipositive. If $D_r$ has no negative Seifert circles, then the Yamada--Vogel algorithm applied to $D_r$
gives a quasipositive braid.
\end{theorem}

Here a negative (or a negatively oriented) Seifert circle of $D_r$ is a circle that is obtained from $D_r$ while applying the Seifert algorithm and that has
rotation number $-1$. In general, $D_r$ might have such circles. A procedure of removing negative Seifert circles by an isotopy of $S^2$,
that is, Step~2 of the Yamada--Vogel algorithm (see Section~\ref{sec:vogelalgorithm})
spoils the assumptions of Theorem~\ref{th:et} used in the proof of Theorem~\ref{main1}. Therefore a more sophisticated algorithm 
is needed. An example of such algorithm is given in Section~\ref{sec:murasugiprzytycki}. This is 
a version of Murasugi--Przytycki algorithm for merging Seifert circles. 
To apply this algorithm we introduce a variant of an index of a diagram (or of a bipartite graph), $\ind_{\#}$ (see Definition~\ref{def:hashindex}).
We prove that $\ind_\#$ is always less than or equal to the number of negative Seifert circles and if a diagram has $\ind_\#$ equal to $k$, then
the Murasugi--Przytycki algorithm can be used $k$ times and consequently $k$ negative Seifert circles can be removed.
We obtain the following generalization of Theorem~\ref{main1}.

\begin{theorem}\label{main2}
Suppose that the link $L_r$ is strongly quasipositive fibered and $n_-=\ind_{\#}(D_r)$, where $\ind_{\#}$ is its doubly negative index. 
Then the Murasugi--Przytycki algorithm of Section~\ref{sec:murasugiprzytycki} followed by the Yamada--Vogel
algorithm applied to $D_r$ gives a quasipositive braid.
\end{theorem}

The proof of Theorems~\ref{main1} and~\ref{main2} uses a result of Etnyre--van Horn-Morris, stated in Theorem~\ref{th:et}. Their result
gives necessary and sufficient conditions for a braid representing a strongly quasipositive and fibered link, to be quasipositive.
The conditions are given in terms of the algebraic length of the braid, the number of strands and the Euler characteristics of the fiber.
Now the Yamada--Vogel algorithm allows us to control the algebraic length of the resulting braid, as well as its number of strands, provided
we know the number of negative Seifert circles, the writhe, and the rotation number of the link diagram at the input. Therefore
we need to control the writhe and the rotation number of the link diagram $D_r$. Finding such control is the main technical part of the present article.

The methods that we use to estimate the writhe and the rotation number of $D_r$ are a variant of Morse theory. 
We study changes of the diagram $D_r$ as $r$ goes from $0$
to infinity. 
When $r$ is small, $L_r$ is an unknot and $D_r$ is a round circle.
As $r$ grows, essentially 
two phenomena can happen. One possibility is that for some $r$, $\mathcal{C}$ ceases to be 
transverse to $S_r$: this corresponds to adding a handle to $\mathcal{C}\cap B_r$ and
the isotopy class of $L_r$ changes. The other case is when the isotopy class of $L_r$ is preserved, but
the diagram $D_r$ 
changes. For a generic $\mathcal{C}$ this change is a Reidemeister move. 

Using the complex structure of 
$\mathcal{C}$ and a detailed analysis in local coordinates, we can restrict possible 
Reidemeister moves.
More precisely, we have the following result, which is the main technical result of the article. Refer to Figure~\ref{fig:R1moves} 
for notation of Reidemeister moves.

\begin{theorem}\label{main3} 
Out of 8 directed Reidemeister $\reid_1$ moves, only 
$\reid_1^{\uparrow\boxplus\oplus}$ and $\reid_1^{\downarrow\boxminus\ominus}$ 
can happen if $\mathcal{C}$ is a complex algebraic curve.
Furthermore, we have restrictions on the handle attachments. The only possible $0$--handle attachment is the one that
increases the winding number $rot(D_r)$, 
while the only possible $1$--handle attachment always decreases the winding number.
\end{theorem}

The quantity
$\rot(D_r)-\writhe(D_r)$ is preserved under 
$\reid_1^{\uparrow\boxplus\oplus}$ and $\reid_1^{\downarrow\boxminus\ominus}$
moves, and both $\rot(D_r)$ and $\writhe(D_r)$ are preserved under the Reidemeister moves $\reid_2$ and $\reid_3$. 
Furthermore $1$--handles decrease the difference $\rot(D_r)-\writhe(D_r)$ and
$0$--handles increase it by one. It follows that $\chi(\mathcal{C}\cap 
B_r)-\rot(D_r)+\writhe(D_r)$ does not depend on $r$, where $\chi$
is the Euler characteristic. In particular we obtain the
following result. 

\begin{theorem}\label{main4}
If $D_r$ is a diagram of $L_r$ obtained by the complex stereographic projection,
then $\chi(\mathcal{C} \cap B_r) = \rot(D_r)-\writhe(D_r)$.
\end{theorem}
\begin{remark}\
\begin{itemize}
\item The logic of proofs of Theorem~\ref{main3} and~\ref{main4} is slightly more complicated than stated here. In fact the first
part of Theorem~\ref{main3} about Reidemeister moves is proved as Theorem~\ref{th:reid1}. Next we prove Theorem~\ref{main4}, finally
the last part of Theorem~\ref{main3} is a corollary of Theorem~\ref{main4}. In this way we can restrict possible handle attachments
without the need to work in local coordinates.
\item In Section~\ref{sec:bennequinsl} we discuss the relation of our results with an analogous result of Bennequin. In
particular we explain that Theorem~\ref{main4} is not a consequence of \cite[Proposition 4]{Ben}.
\end{itemize}
\end{remark}

In Theorem~\ref{main3} we study which variants of the $\reid_1$ move can occur on $D_r$ as $r$ changes. 
In theory it is possible that one can also control variants of the $\reid_2$ and $\reid_3$ moves. 
Such a control should use slightly different methods than 
those used in the present article,
because $\reid_2$ and $\reid_3$ are not local moves. If some constraints are obtained, one might potentially find further restrictions on the diagram $D_r$,
in terms of more sophisticated diagram invariants, like the Hass--Nowik invariant, see \cite{HN,Suw}. These restrictions might in turn be enough to show
that for any diagram $D_r$ obtained as a complex stereographic projection of a $\Cc$-link satisfies $\ind_\#(D_r)=n_-(D_r)$, that is, that $D_r$ satisfies
the hypotheses of Theorem~\ref{main2}. This would lead to an explicit algorithm for transforming a strongly quasipositive fibered $\Cc$-link into a quasipositive
braid.

The structure of the paper is the following: Section~\ref{sec:contact} sets up the notation and recalls various transversality notions. 
Section~\ref{sec:stereo} defines the complex stereographic projection and sets up basic properties of diagrams of $\Cc$-links. 
Sections~\ref{sec:changes} and \ref{sec:euler} are  the heart of the article. They contain results about possible changes
of the diagram $D_r$ as $r$ increases. We prove Theorem~\ref{main3} and \ref{main4}. We also discuss its relation with the results of Bennequin.
Section~\ref{sec:diagramsandbraids} recalls the Yamada--Vogel algorithm and then proves Theorem~\ref{main1}. We turn our attention to
negative Seifert circles and show how to apply the Murasugi--Przytycki construction to eliminate them in some cases. 
Section~\ref{sec:diagramsandbraids} ends with the proof of Theorem~\ref{main2}. Section~\ref{sec:proofs} contains proofs of many
lemmas used in Section~\ref{sec:changes}: some of the proofs require tedious calculations in local coordinates, so they are deferred to the last
section of the article.

\begin{ack}
The authors would like to thank to Sebastian Baader, Michel Boileau, Matthew Hedden, Patrick Popescu-Pampu,
J\'ozef Przytycki and Pawe\l{} Traczyk for helpful discussions.
The authors were supported by the POLONIUM program, and by the ANR project ``SUSI'' (ANR-12-JS01-0002-01).
The second author is supported by the National Science Center grant 2016/22/E/ST1/00040.
\end{ack}

\section{Transverse knots and contact structures}
\label{sec:contact}

In this section we set up some notation and recall two notions of transversality.

\subsection{Review of notation on $\Cc^2$}

Coordinates of points in $\Cc^2$ are usually denoted by $(x,y)$. Identifying 
$\Cc^2$ with $\Rr^4$ means that we write
$(x,y)=(x_1,x_2,y_1,y_2)$, where $x=x_1+\ii x_2$ and $y=y_1+\ii y_2$.
Two scalar products are defined on $\Rr^4 \simeq \Cc^2$:
\begin{itemize}
  \item \defi{The real scalar product}:
  $$\left\langle a | b \right\rangle_\Rr = \sum a_i\cdot b_i \quad \text{ for }
  a=\begin{pmatrix}a_1\\a_2\\a_3\\a_4\end{pmatrix}, \  
b=\begin{pmatrix}b_1\\b_2\\b_3\\b_4\end{pmatrix} \in \Rr^4$$ 
  
  \item \defi{The Hermitian scalar product}:
  $$\left\langle a | b \right\rangle_\Cc = \sum \alpha_i\cdot \overline{\beta_i} 
\quad \text{ for }
  a=\begin{pmatrix}\alpha_1\\\alpha_2\end{pmatrix}, \  
b=\begin{pmatrix}\beta_1\\\beta_2\end{pmatrix} \in \Cc^2$$ 
\end{itemize}
The two are related by the following relation:
$$\left\langle a | b \right\rangle_\Rr = \Re \big( \left\langle a | b 
\right\rangle_\Cc \big) $$

\subsection{Smooth transversality} 

Suppose $\mathcal{C}=\{f=0\}$ is a complex curve in $\Cc^2$, where $f$ is a 
reduced polynomial.
Fix $r>0$ and consider $S_r$, the sphere of radius $r$ and center $0$.
We  have:
$$v \in T_{z}S_r \iff \left\langle v | z \right\rangle_\Rr = 0 \quad \text{ 
and } \quad
v \in T_{z}\mathcal{C} \iff  \left\langle v | \grad f(z) \right\rangle_\Cc = 
0$$
The gradient $\grad f$ is \defi{Milnor gradient}:
$$\grad f(x,y) = \left(\overline{\frac{\partial f}{\partial x}} (x,y), 
\overline{\frac{\partial f}{\partial y}} (x,y)\right).$$
The following result is standard, see e.g. \cite{Mi}.
\begin{lemma}
The intersection of $\mathcal{C}$ with $S_r$ at $z$ is (smoothly) transverse if and only if
the two vectors $\grad f$ and $z$ are linearly independent over $\Cc$.
\end{lemma}

\begin{definition}
If $\mathcal{C}$ is transverse to $S_r$, then the intersection of $\mathcal{C}$ with $S_r$ is called a $\Cc$-link and denoted by $L_r$.
In general, a link $L$ is called a \emph{$\Cc$-link}, if it arises as a transverse intersection of a sphere in $\Cc^2$ with a complex curve.
\end{definition}
\begin{remark}
Sometimes a $\Cc$-link is called a \emph{transverse link}. In this paper we use the word $\Cc$-link to avoid confusion with links that are
transverse to the standard contact structure on $S_r$; see below.
\end{remark}
\subsection{Contact transversality} \label{sec:contact2}

The plane field on $TS_r$:
$$H_z = T_{z}S_r \cap \ii T_{z}S_r,$$
where $\ii$ means the complex multiplication by $\sqrt{-1}$
on the tangent space $T_{z}S_r$, defines a standard contact structure on 
$S_r$. In coordinates
$(x,y)$ in $\Cc^2$ define a $1$-form:
\begin{equation}\label{eq:contact}
\alpha = x_1\dd x_2 - x_2\dd x_1 + y_1 \dd y_2 - y_2\dd y_1.
\end{equation}
On $S_r$ we have $H_z=\ker\alpha$.
We also have:
$$v \in H_z 
\iff \left\{\begin{array}{l} 
  v \in T_{z}S_r \\ 
  \ii v \in T_{z}S_r
\end{array}\right.
\iff \left\{\begin{array}{l} 
  \left\langle v | z \right\rangle_\Rr = 0 \\ 
  \left\langle \ii v | z \right\rangle_\Rr = 0
\end{array}\right.
\iff  \left\langle v | z \right\rangle_\Cc = 0$$

\begin{definition}[Contact transversality] A link  $L_r \subset S_r$ is 
\defi{(contact) transverse} to the standard contact structure
on $S_r$ at $z$ if and only if $T_zL_r\cap H_z=0$.
\end{definition}

The following result is standard.
\begin{proposition}
If a link $L_r = \mathcal{C} \cap S_r$ is the smoothly transverse intersection of the 
algebraic curve $\mathcal{C}$ and the sphere $S_r$ at any point of $L_r$, then the link 
$L_r$ is contact transverse to the standard contact structure of $S_r$. Moreover,
$L_r$ is \emph{positively transverse}, that is if $z\in L_r$ and $v\in T_zL_r$ is a positive vector
(that is, agreeing with the orientation of $L_r$), then 
\begin{equation}\label{eq:positive}
\alpha(v)>0.
\end{equation}
\end{proposition}

\begin{remark}
Both notions of transversality, smooth and contact, are widely used throughout the paper. In order to avoid confusion, we
use more precise terminology `smooth transversality' and `contact transversality', even if it might be clear from the context.
A word 'non-transversality' always refers to lack of smooth transversality.
\end{remark}
%
%
%
%

\section{The complex stereographic projection}
\label{sec:stereo}

\subsection{The construction}
Consider now $\Cc\times\Rr\simeq \Rr^3$. We choose coordinates $X,Y,Z$ on $\Rr^3$
so that $X+\ii Y$ is a coordinate on $\Cc$. 
\begin{definition}
The map 
$$\Psi_r : S_r \setminus\{(0,-r)\} \to \Rr^3 \simeq \Cc \times \Rr \qquad 
(x,y) \mapsto \left(\frac{x}{r+y}, \frac{-\ii r (y-\ol y)}{2|r+y|^2}\right)$$
is called the \defi{complex stereographic projection}.
\end{definition}

The following result is well-known, we refer to the survey of Geiges \cite{Ge}.

\begin{lemma}
\label{lem:psir}
The map $\Psi_r$ has the following properties.
\begin{itemize}
  \item $\Psi_r$ is an orientation preserving diffeomorphism;
  
  \item $\Psi_r$ sends the standard contact structure $\alpha$ 
  on the sphere $S_r$ to the standard contact structure of $\Rr^3$.
  More precisely, if $\beta  = \dd Z - (Y\dd X-X\dd Y)$ then
  $$\Psi^*(\beta) = \frac{1}{|r+y|^2} \; \alpha$$
  
\end{itemize}
\end{lemma}

\begin{example}
If $L_r = \{y=0\} \cap S_r$, then
$\Psi_r (L_r) = S^1_1 \times \{0\} \subset \Cc\times \Rr $.
\end{example}

\begin{corollary}\label{cor:istransverse}
If $\Psi_r(L_r)$ is the image of a $\Cc$-link, then for all points $z'\in\Psi_r(L_r)$, and positive vectors $v'\in T_z\Psi_r(L_r)$
(positive means that the direction of the vector agrees with the orientation of $\Psi_r(L_r)$), we have $\beta(z',v')>0$.
In other words along the link image:
$$\dd Z > Y\dd X-X\dd Y$$
\end{corollary}

\begin{proof}
For a point $z=(x,y) \in S_r$ and a vector $v\in T_{z} S_r$, we denote
$z' = \Psi_r(z)$ and $v' = \dd\Psi_r(z)(v)$.
For $v\in T_{z} L_r$:
$$\beta(z',v')=\beta\big(\Psi_r(z),\dd\Psi_r(z)(v) \big)
= \Psi_r^*(\beta)(z,v) = \frac{1}{|r+y|^2}\alpha(z,v)>0.$$
\end{proof}

It might happen that the curve $\mathcal{C}$ intersects the real half-line 
$(0,0,-t,0)$, $t\ge 0$. Then $\Psi_r$ is not defined at some points of $L_r$. To avoid this
situation observe that the
set of possible half-lines has real dimension $3$, and $\mathcal{C}$ has dimension 
$2$. Suppose $0\notin\mathcal{C}$ (this is a generic condition). As each point of $\mathcal{C}$ belongs to precisely
one half-line, by dimension counting argument there exists a half-line
that does not intersect $\mathcal{C}$ at all. We can rotate the coordinate 
system on $\Cc^2$ by an $SU(2)$ matrix so that $(0,0,-t,0)$
is a half-line not intersecting $\mathcal{C}$. 

From now on we shall assume that 
$(0,0,-t,0)$ does not meet $\mathcal{C}$.

\subsection{Stokes formula and negative circles in the projection}

To a $\Cc$-link $L_r$, we associate its \defi{image} $\tilde L_r = \Psi_r(L_r)$ in 
$\Rr^3$ 
by the complex stereographic projection. 
Define the projection $\pi' : \Rr^3 \to \Rr^2$, $(X,Y,Z)\mapsto (X,Y)$
and the composition $\pi_r = \pi' \circ \Psi_r$.
Then we associate to  $L_r$ its \defi{diagram} 
$D_r= \pi_r(L_r) = \pi'\circ\Psi_r(L_r)$ in $\Rr^2$.

We are going to exploit the property of $L_r$ stated in 
Corollary~\ref{cor:istransverse}. The key point will be the Stokes formula.
Let $D_r$ be a link diagram considered as an oriented curve in 
$\Rr^2$. We need some terminology.

\begin{definition}\
\begin{itemize}
\item A piecewise smooth closed curve $\gamma\subset D_r$, whose orientation is 
compatible with $D_r$, is called a \defi{circuit}.
\item The complement in $\Rr^2$ of a circuit $\gamma$ is a union of regions 
$A_\infty,A_1,\ldots,A_k$, where $A_\infty$ is the only unbounded region. For 
$j=1,\ldots,k$ we set $\ind_j(\gamma)$ to
be the index of $\gamma$ relative to a point in the interior of $A_j$. This can 
be defined as the signed intersection number of $\gamma$ and any half-line in 
$\Rr^2$, whose endpoint is in
the interior of $A_j$.
\item A circuit is called \defi{simple} if $\gamma$ is a simple closed curve. 
This amounts to saying that $k=1$ and $\ind_1(\gamma)=\pm 1$.
\item A circuit is \defi{positive} if for all finite $j$, $\ind_j(\gamma)>0$. It is 
negative if $\ind_j(\gamma)<0$ for all finite $j$.
\item A circuit is called a \defi{Seifert circuit} if it is simple and at each double point $w$ 
of $D_r$ that belongs to $\gamma$, $\gamma$ makes a turn. In particular, $\gamma$ is not smooth at any double point it passes through.
\end{itemize}
\end{definition}

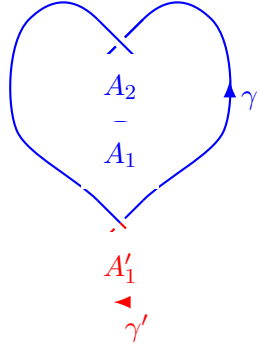
\begin{figure}[H]
  \input{pic-link-26.tex}
  \caption{A positive non-simple circuit $\gamma$ (in blue) with $A_1$ of index $1$ and $A_2$ of index $2$.
  A negative simple circuit $\gamma'$ (in red, dashed line) with $A_1'$ of index $-1$.\label{fig:circuit}}
\end{figure}
\begin{remark}
There is a subtle difference between a Seifert circuit and a Seifert circle. The latter is a part of the smoothed diagram, while
the first is a part of the diagram. 
\end{remark}
For a smooth arc $\alpha\subset D_r$, there is a unique lift 
$\tilde\alpha\subset \tilde L_r\subset\Rr^3$
of $\alpha$ to $\tilde L_r$. Suppose $\gamma$ is a circuit. 
It is a union of smooth arcs $\alpha_1,\ldots,\alpha_n$. 
Order the arcs so that the endpoint of 
$\alpha_i$ is the beginning of $\alpha_{i+1}$ (with
$\alpha_{n+1}$ understood as $\alpha_1$).
The lift $\tilde\gamma$ 
is a disjoint union of smooth oriented connected arcs 
$\tilde\alpha_1,\ldots,\tilde\alpha_n$.

\begin{definition}
Let $w_i$ be the endpoint of $\alpha_i$. 
The \defi{jump} $\lambda_i$ of $\gamma$ at $w_i$, is the difference of the $Z$ 
coordinates of $\tilde\alpha_{i+1}$ and $\tilde\alpha_{i}$ at the point $w_i$.
\end{definition}

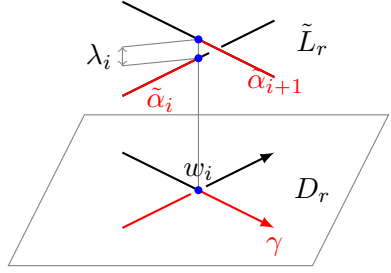
\begin{figure}[h]
  \input{pic-link-27.tex}
  \caption{The jump.\label{fig:jump}}
\end{figure}

The fundamental consequence of contact transversality of $L_r$ is the following.
\begin{proposition}\label{prop:stokes}
For any circuit $\gamma$ of $D_r$ we have
\[\sum_{i=1}^n\lambda_i< 2\mathcal{A}(\gamma),\]
where
\[\mathcal{A}(\gamma)=\sum_{i=1}^k\ind_j(\gamma)\area(A_j),\]
and $A_1,\ldots,A_k$ are the bounded connected components of the complement 
$\Rr^2\setminus\gamma$.
\end{proposition}
\begin{proof}
As $\dd(Y\dd X-X\dd Y)=-2\dd X\wedge \dd Y$, by Stokes theorem
\[\int_{\gamma} Y\dd X-X\dd 
Y=-2\sum_{i=1}^k\ind_j(\gamma)\area(A_j)=-2\mathcal{A}.\]
The left hand side can be transformed to
\[\int_{\gamma} Y\dd X-X\dd Y=\sum_{i=1}^n\int_{\alpha_i} Y\dd X-X\dd 
Y=\sum_{i=1}^n\int_{\tilde\alpha_i}Y\dd X-X\dd Y.\]
By Corollary~\ref{cor:istransverse} 
\[\int_{\tilde\alpha_i}Y\dd X-X\dd Y<\int_{\tilde\alpha_i}\dd Z.\]
The integral on the right hand side is the difference of the $Z$ coordinate of the endpoint of 
$\tilde\alpha_i$ and its starting point. As $\gamma$ is a closed curve, the
sum of those differences and the jumps of $\gamma$ is equal to zero, that is to 
say,
\[\sum_{i=1}^n\int_{\tilde\alpha_i} \dd Z+\sum_{i=1}^n \lambda_i=0.\]
This means that
\[\sum_{i=1}^n \lambda_i=-\sum_{i=1}^n\int_{\tilde\alpha_i} 
\dd Z<-\sum_{i=1}^n\int_{\tilde\alpha_i}Y\dd X-X\dd Y=2\mathcal{A}.\]
\end{proof}

As a simple application, we have the following fact that we will need later.
\begin{corollary}\label{cor:negative}
If $\gamma$ is negative, then $\gamma$ must have a crossing with negative jump.
\end{corollary}
In particular, we have:
\begin{corollary}\label{nonegativecircles}
A negatively oriented unknot pictured in Figure~\ref{fig:negativecircle}
is \textbf{not} the projection of any $\Cc$-link. It is not a component of the complex projection of any $\Cc$-link either. 
\end{corollary}
\begin{figure}[h]
  \input{pic-link-25.tex}
  \caption{Unknot projection with negative 
orientation.}\label{fig:negativecircle}
\end{figure}
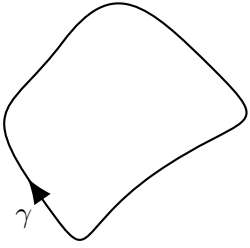

\subsection{Complex stereographic projection vs. standard stereographic projection}\label{sec:standardstereo}
We will show on an example that there is an advantage of working with complex stereographic projection over working
with a standard stereographic projection.

The standard stereographic projection $\wt{\Psi}_r\colon S^3_r\setminus\{(0,0,-r,0)\}\to\Rr^3$ is given by the formula
\[(x_1,x_2,y_1,y_2)\mapsto \left(\frac{x_1}{y_1+r},\frac{x_2}{y_1+r},\frac{y_2}{y_1+r}\right).\]
This projection can be composed with a projection to the first two coordinates, which we denote by $\wt{\pi}_r$.
Let $\wt{D}_r$ be the diagram of $\wt{\Psi}_r(L_r)$ under the projection $\wt{\pi}_r$.

It turns out that Proposition~\ref{prop:stokes} is no longer true for $\wt{D}_r$.
An example, calculated numerically by Sagemath \cite{sage},
is given by a curve $\mathcal{C}$ given in parametric form\footnote{with
$a_3=0.2563-0.1587\ii$, $a_2=-0.1048+0.0393\ii$, $a_1=-0.2986-0.4498\ii$, $a_0=-0.2052+0.0618\ii$, $b_3=-0.4786-0.2976\ii$, $b_2=0.0099+0.1586\ii$,
$b_1=-0.4694-0.1366\ii$ and $b_0=-0.244+0.3914\ii$ and $r=0.8$} by $x(t)=a_3t^3+a_2t^2+a_1t+a_0$, $y(t)=b_3t^3+b_2t^2+b_1t+b_0$. We consider the link $L=L_{r}$ 
obtained by intersecting $\mathcal{C}$ with the sphere of radius~$r$.
The link $L$ is trivial. 
The diagram $\wt{D}=\wt{\pi}_{r}(\wt{\Psi}_{r}(L))$
is drawn in Figure~\ref{fig:stupidlink}. The two crossings of $\wt{D}$ are positive.
There are three Seifert circuits of $\wt{D}$, but one of them has orientation opposite to the two others. It follows that regardless of
the orientation of $\wt{D}$, it has a negatively oriented Seifert circuit with positive jumps only.

Further computer experiments with this $\mathcal{C}$ indicate that Theorem~\ref{th:reid1} does not hold for the standard stereographic projection either.

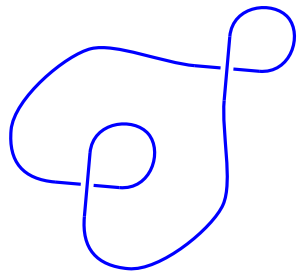
\begin{figure}[h]
  \input{pic-link-32.tex}
\caption{The diagram, obtained by the projection map $\wt{\pi}_{r}$, of the link $L$ from Section~\ref{sec:standardstereo} under the standard stereographic projection.}\label{fig:stupidlink}
\end{figure}


\section{Changes of the diagram when $r$ changes}\label{sec:changes}

The radius $r$ of the sphere $S_r$ can be considered as a parameter. In this way obtain a 
family of links $L_r$ and
of diagrams $D_r$. It is known that if for given $r_1$, $r_2$ the curve
$\mathcal{C}$ is smoothly transverse to all spheres $S_r$ for $r\in[r_1,r_2]$, then 
the links $L_{r_1}$ and $L_{r_2}$
are isotopic. The diagrams $D_{r_1}$ and $D_{r_2}$ do not have to be isotopic; 
there might be some
Reidemeister moves relating the first one to the second one. While the moves $\reid_2$
and $\reid_3$ are not local and therefore harder to control, 
we will identify the $\reid_1$ moves as occurring precisely when the map $\pi_r$ 
restricted to $L_r$ has critical points. First we describe the non-transversality 
points in greater detail.

Throughout this section we shall assume that $\mathcal{C}$ is generic. A precise
statement of genericity is given in Section~\ref{sec:afewwords}. We prove there
that genericity is a dense condition so it 
can be achieved by a small perturbation of $\mathcal{C}$. There is no loss of generality when restricting only to generic curves,
since a diagram $D_r$ of any $\Cc$-link can be perturbed to a diagram of a $\Cc$-link 
with $\mathcal{C}$ generic.

\subsection{The non-transversality points}

The points of non-transversality can be easily computed. To this end let us 
introduce the function $Jf(x,y)$ by the formula:
\begin{equation}\label{eq:Jf}
Jf(x,y) = 
\left|\begin{matrix} \frac{\partial f}{\partial x} & \frac{\partial f}{\partial 
y}\\ \ol{x} &\ol{y}\end{matrix}\right|.
\end{equation}
The intersection $\{Jf(x,y)=0\}$ with $\mathcal{C}$ is the set of points 
where the function $|x|^2+|y|^2$, that is ``square of the distance to origin'', restricted to $\mathcal{C}$ 
has a critical point. We 
assume that the intersection of the two curves 
$\mathcal{C}$ and $\{Jf(x,y)=0\}$ is smoothly transverse.
 
Let $z=(x_0,y_0)$ belong to $\mathcal{C}\cap\{Jf(x,y)=0\}$. When $r$ crosses the 
critical value $\|z\|$, one of the two situations
can occur: either $z$ is a local minimum of the distance function on 
$\mathcal{C}$, and then $L_{\|z\|+\varepsilon}$ differs
from $L_{\|z\|-\varepsilon}$ by adding an unlinked unknot (here $\varepsilon>0$ is a 
small parameter); 
or $z$ is a saddle point, so $L_{\|z\|+\varepsilon}$ arises from 
$L_{\|z\|-\varepsilon}$ by a single 1--handle attachment. As $|x|^2+|y|^2$ is plurisubharmonic,
it does not have any local maxima on the complex curve $\mathcal{C}$.

\begin{definition}
An \emph{$\Oo$ move} on a link $L$ adds to $L$ an unlinked unknot, which becomes a round
circle disjoint from the rest of the diagram under the complex stereographic projection. 
An \emph{$\Ii$  move} adds a 1--handle to $L$. If we keep track
of the orientations, we can specify an $\Oo^{\oplus}$ and an $\Oo^{\ominus}$ move 
consisting of adding a positive or a negative circle, respectively;
see Figure~\ref{fig:oooplus}. We can also distinguish two types of an $\Ii$ move, 
that is, $\Ii^\oplus$ and $\Ii^\ominus$; see  Figure~\ref{fig:iiiplus}.
\end{definition}
\begin{figure}[h]
\input{pic-link-11.tex}
\caption{The positive/negative orientation.}\label{fig:oooplus}
\end{figure}
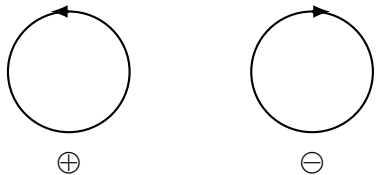

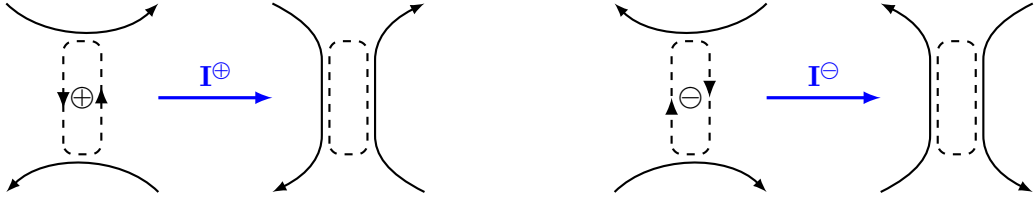
\begin{figure}[h]
\input{pic-link-12.tex}
\caption{The positive/negative $1$-handle attachment.}\label{fig:iiiplus}
\end{figure}

\begin{proposition}\label{prop:localindex}
The local intersection index of $\mathcal{C}$ and $\{Jf(x,y)=0\}$ at $z$ is $1$ 
if and only if $z$ is a saddle point (of index $1$). 
It corresponds to a move $\Ii$, an attachment of a $1$-handle.
  
If the intersection index is $-1$, then $z$ 
is a local minimum (of Morse index $0$). It corresponds to a move 
$\Oo$, the birth of a component.
\end{proposition}
The proof of Proposition~\ref{prop:localindex} is technical and is postponed 
until Section~\ref{sec:proofoflocalindex}.

The fact that $L_r$ is positively contact transverse implies immediately the following fact.
\begin{lemma}\label{noominus}
The move $\Oo^{\ominus}$ is impossible.
\end{lemma}
\begin{proof}
The move creates an unknotted negatively oriented circle in a diagram $D_r$. 
But this is impossible by Corollary~\ref{nonegativecircles}.
\end{proof}
In Corollary~\ref{cor:ominus}  we will show that $\Ii^{\ominus}$ is also impossible.

\begin{lemma}
\label{lem:critproj}
The critical points of the projection $\pi_r : L_r \to \Rr^2$
defined by $\pi_r = \pi'\circ \Psi_r$ are:
\begin{itemize}
  \item the points of non-transversality of $\mathcal{C}$ with $S_r$ given by the 
equation 
  $\begin{vmatrix}
\frac{\partial f}{\partial x}  & \frac{\partial f}{\partial y} \\
\ol x & \ol y \\
\end{vmatrix} = 0$;

  \item and the points satisfying 
  $x \frac{\partial f}{\partial x} + (r+y)\frac{\partial f}{\partial y} = 0$.
\end{itemize}
\end{lemma}

The proof of Lemma~\ref{lem:critproj} is given in 
Section~\ref{sec:proofofcritproj}.

\begin{example}\label{ex:weirdexample}
Let $f(x,y) = ax+by+c$. 
\begin{enumerate}
  \item The points of non-transversality are given by equations:
$ax+by+c=0$  and  $a\ol y - b \ol x = 0$.
There is a unique solution:
$x = \frac{-\ol a c}{|a|^2+|b|^2}$, $y = \frac{-\ol b c}{|a|^2+|b|^2}$
for $r = \frac{|c|}{\sqrt{|a|^2+|b|^2}}$.

  \item The other critical points of $\pi_r$ are given by equations
  $x \frac{\partial f}{\partial x} + (r+y)\frac{\partial f}{\partial y} = 0$
  and $f(x,y)=0$
  i.e.
  $$ax+by + br= 0 \quad \text{ and }\quad ax+by+c=0$$
  For $r \neq \frac{c}{b}$, there are  no critical points.
  For $r = \frac{c}{b}$ (provided $b\neq0$) 
  any point of $L_r$ is a critical point of $\pi_r$.
  In fact the projection of the whole link $L_r$ by $\pi_r$ is
  just the point $-\frac{c}{a}$ (provided $a\neq0$).  
  
\end{enumerate}

The fact that for $r=\frac{c}{b}$ the whole of $L_r$ is mapped to a point reflects
the lack of genericity of the polynomial $f$. It is related to the fact
that the second derivative of $f$ vanishes at critical points of $\pi_r$ (actually
it vanishes everywhere). See Section~\ref{sec:afewwords} for a discussion of genericity issues.
\end{example}


\subsection{The first Reidemeister move}

It will turn out that the critical points of $\pi_r$ given by 
$x \frac{\partial f}{\partial x} + (r+y)\frac{\partial f}{\partial y} = 0$ in
Lemma~\ref{lem:critproj} correspond to the first Reidemeister move on the 
diagram. Of course, as Example~\ref{ex:weirdexample} suggests, we will have to
impose a genericity condition on $\mathcal{C}$.

The first Reidemeister move $\reid_1$ is a creation (or disappearance)
of a loop in a diagram of the link.
\begin{figure}[h]
  \input{pic-link-08.tex}
  \caption{The first Reidemeister move.}
\end{figure}
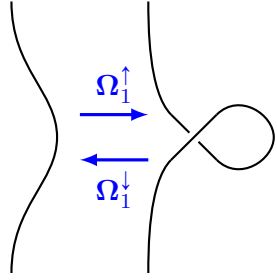
As in \cite{Suw} we distinguish between eight types of moves, according to:
\begin{itemize}
  \item the creation/disappearance of a loop: $\uparrow/\downarrow$
  \item the sign  plus/minus of the crossing: $\boxplus/\boxminus$
  \item the positive/negative orientation of the loop: $\oplus/\ominus$
\end{itemize}

It is clear from the description, how the writhe and the rotation
number of a diagram change upon performing any of these moves. The data is gathered in Table~\ref{tab:changes}.

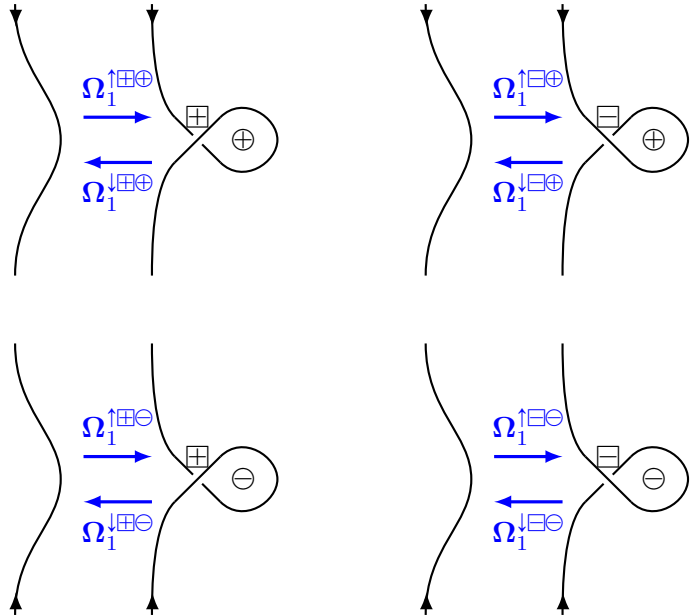
\begin{figure}[h]
  \input{pic-link-09.tex}
  \caption{The eight types of the first Reidemeister move.}\label{fig:R1moves}
\end{figure}

The following result proves the first part of Theorem~\ref{main3} from the introduction.
\begin{theorem}
\label{th:reid1}
Suppose that $(x_0,y_0) \in L_{r_0}$ is a critical point of the projection  
$\pi_{r_0}'$, but $(x_0,y_0)$ is a point of a smooth transverse intersection of $\mathcal{C}$
and $S_{r_0}$.
Then for $r$ near $r_0$ (going from $r_0-\epsilon$ to $r_0+\epsilon$, 
$\epsilon>0$) 
the topological type of the link is unchanged, but
the projection $D_r = \pi_r(L_r) \subset \Rr^2$ changes by a Reidemeister
move of type $\reid_1^{\uparrow\boxplus\oplus}$ or 
$\reid_1^{\downarrow\boxminus\ominus}$.

\end{theorem}

Both moves $\reid_1^{\uparrow\boxplus\oplus}$ or 
$\reid_1^{\downarrow\boxminus\ominus}$
can actually happen.
The proof of Theorem~\ref{th:reid1} is given in Section~\ref{sec:proofofreid1}.
The converse is also true, a Reidemeister move of type $\reid_1$
implies a critical point of the projection. More precisely we have the following result.

\begin{proposition}
If a family of diagrams $D_r$ makes a Reidemeister move $\reid_1$
for some radius $r_0$, then the projection $\pi_{r_0}$ has a critical point.
\end{proposition}

\begin{proof}
We prove the proposition for the $\reid_1^{\downarrow}$ move, for the $\reid_1^{\uparrow}$
the proof is essentially the same.
The local situation for the move $\reid_1^{\downarrow}$ is the following:
at $r_0-\epsilon$ the link diagram $D_r$ has a loop; at $r_0+\epsilon$ 
the link diagram is locally a smooth arc.
At each $r \in [r_0-\epsilon,r_0)$ we define two points on each loop:
\begin{itemize}
  \item $Q_1(r)$, such that the tangent at this point is horizontal;
  \item $Q_2(r)$, such that the tangent at this point is vertical.  
\end{itemize}

  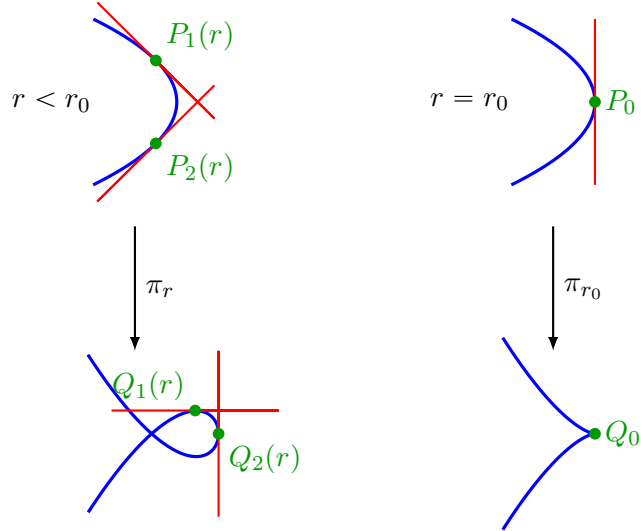
\begin{figure}[h]
    \input{pic-link-24.tex}
    \caption{The horizontal/vertical tangents and their projection.}
  \end{figure}
  
The fact that at $r_0$ the loop disappears, implies that
$Q_1(r) \to Q_0$ and $Q_2(r) \to Q_0$, as $r\to r_0$.
Let $P_1(r)$ and $P_2(r)$ be points of the link $L_r$ such that
$\pi_r(P_i(r)) = Q_i(r)$. And let $P_0$ such that
$\pi_r(P_0) = Q_0$.
The tangent at $P_0$ is the limit of the tangents at $P_1(r)$ as
$r \to r_0$. This tangent at $P_0$ is also the limit of the tangents
at $P_2(r)$ as $r \to r_0$.

Assume by contradiction, that $\pi_{r_0}$ has no critical points;
we may also suppose that $\pi_{r}$ has no critical points
for $r\in [r_0-\epsilon,r_0)$.
The tangent at $P_1(r)$ is sent to the tangent at $Q_1(r)$,
by definition of $Q_1(r)$ this last tangent is a horizontal line.
The tangent at $P_2(r)$ is sent to a vertical line.
The tangent at $Q_0$ exists and is the image of the tangent
at $P_0$ because $\pi_{r_0}$ is assumed to have no critical points.
It implies that the tangent at $Q_0$ is both the limit of horizontal lines
(the tangents at $Q_1(r)$) and vertical lines (the tangents at $Q_2(r)$).
This yields a contradiction. 
\end{proof}


\section{Euler characteristic of $\Cc$-links}\label{sec:euler}

To a diagram $D$ of a link $L$ we associate two numbers:
\begin{itemize}
  \item the \defi{writhe}, $\writhe(D)$, which is the sum of positive 
crossings~$\boxplus$, minus the negative crossings~$\boxminus$ (see Figure~\ref{fig:writhe}); this number does not
change, if the orientation of $L$ is reversed, but it can change if one reverses the orientation of some, but not all, components of
$L$;
  
  \item the \defi{winding number} $\rot(D)$ 
  (or the \defi{rotation number}) is the degree of the Gauss map from the diagram $D$ to $S^1$
  defined by the unit tangent vectors. Informally, $\rot(D)$ 
  counts the number of turns 
  made when going along the diagram. If $D'$ is equal to $D$ with reversed orientation, then $\rot(D')=-\rot(D)$.
\end{itemize}

  \begin{figure}[h]
    \input{pic-link-10.tex}
    \caption{The positive/negative crossing. \label{fig:writhe}}
  \end{figure}
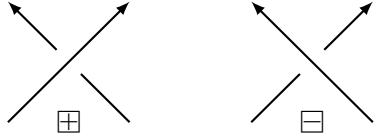
We return to $\Cc$-links. Define
\begin{equation}\label{eq:whatisLr}
\chi(L_r) = \chi(\mathcal{C}_r),
\end{equation}
where $\mathcal{C}_r=\mathcal{C}\cap B_r$.
\begin{remark}
If $L_r$ is strongly quasipositive, then $\mathcal{C}_r$ minimizes the genus among all smooth surfaces in $B_r$
with boundary $L_r$ (this is a result of Kronheimer and Mrowka, see \cite{BF} for an exposition in the language of $\Cc$-links). It follows that $\chi(L_r)=2-2g_4(L_r)-\# L_r$, where $\# L_r$ is the number
of components of $L_r$ and $g_4(L_r)$ is the smooth four-genus, therefore $\chi(L_r)$ depends only on $L_r$ and not on $\mathcal{C}_r$.
If $L_r$ is quasipositive but not strongly quasipositive, then the genus of $\mathcal{C}_r$ might be actually bigger than the four-genus of $L_r$.
\end{remark}

The following result proves Theorem~\ref{main4} of the introduction.
\begin{theorem}
\label{th:charlink}
$$\chi(L_r)=\rot(D_r)-\writhe(D_r)$$  
\end{theorem}

\begin{proof}
Let us define
\[\Theta(r)=\chi(L_r)-\big(\rot(D_r)-\writhe(D_r)\big).\]
Our goal is to show that $\Theta\equiv 0$.

\textbf{Step 1. }\emph{Proof under an extra assumption.}

Start with $r$ close to zero. If $\mathcal{C}$ is generic, we may assume it does not
pass through the origin, so $D_r$ is empty. It has $\rot=\writhe=\chi=0$, so $\Theta=0$. 

\smallskip
The extra assumption is that $\Theta(r)=0$ for $r$ sufficiently large.

\smallskip
\begin{table}
\renewcommand{\arraystretch}{1.5}
\begin{tabular}[t]{ccccc} 
   \qquad \qquad    & \quad $\chi$ \quad & \quad $\rot$ \quad  & \quad $\writhe$ 
\quad  & \quad $\chi - (\rot - \writhe)$ \quad\\ \hline\hline
   $\Oo^{\oplus}$           & $+1$ & $+1$ & $0$ & $0$    \\ \hline
   $\Oo^{\ominus}$           & $+1$ & $-1$ & $0$ & $+2$    \\ \hline
   $\Ii^{\oplus}$ & $-1$ & $-1$ & $0$ & $0$    \\ \hline    
   $\Ii^{\ominus}$  & $-1$ & $+1$ & $0$ & $-2$   \\ \hline    
   $\reid_1^{\uparrow\boxplus\oplus}$ or $\reid_1^{\downarrow\boxminus\ominus}$    & $0$  & $+1$  & $+1$ & $0$    \\ 
\hline  
   $\reid_1^{\uparrow\boxminus\oplus}$ or $\reid_1^{\downarrow\boxplus\ominus}$    & $0$  & $+1$  & $-1$ & $-2$    \\ 
\hline 
   $\reid_1^{\uparrow\boxplus\ominus}$ or $\reid_1^{\downarrow\boxminus\oplus}$    & $0$  & $-1$  & $+1$ & $+2$    \\ 
\hline 
   $\reid_1^{\uparrow\boxminus\ominus}$ or $\reid_1^{\downarrow\boxplus\oplus}$   & $0$  & $-1$  & $-1$ & $0$    \\ 
\hline    
   $\reid_2$       & $0$  & $0$  & $0$ & $0$    \\ \hline    
   $\reid_3$       & $0$  & $0$  & $0$ & $0$    \\     
\end{tabular}
\caption{The changes of various quantities associated to $D_r$ under Reidemeister moves and handle attachments.
  The meaning of the table is the following:
  if for example $D_{r+\varepsilon}$ is the link obtain from $D_{r-\varepsilon}$ after move
  $\Ii^{\oplus}$, then
  $\rot(D_{r+\varepsilon})=\rot(D_{r-\varepsilon})-1$.
  }\label{tab:changes}
\end{table}

In Table~\ref{tab:changes} we gathered the changes of various quantities upon performing a given move. 
Verifying the values in the table is routine and follows directly from the definition of the moves.

From the assumption it follows by Table~\ref{tab:changes}, Lemma~\ref{noominus}
and Theorem~\ref{th:reid1} that no allowed move can actually increase $\Theta$.
Therefore, if $\Theta=0$ for small $r$ and for $r$ close to infinity, then we infer
that $\Theta\equiv 0$.

\smallskip
\textbf{Step 2.} \emph{Perturbing $\mathcal{C}$ at infinity.}

Fix $r_0>0$. Suppose that $d=\deg f$. Take $\wt{f}(x,y)=f(x,y)+\eta(x^\delta+y^\delta)$,
where $\delta>d$ and $\eta\neq 0$ is a very small complex number. Denote by $\wt{\mathcal{C}}$ the zero set of $\wt{f}$
and let $\wt{\Theta}$ be the $\Theta$ function for $\wt{\mathcal{C}}$.

If $\eta$ is sufficiently small, the link $\wt{L}_{r_0}=\wt{\mathcal{C}}\cap S_{r_0}$ is isotopic to $L_{r_0}$ and 
the corresponding diagrams $D_{r_0}$ and $\wt{D}_{r_0}$ are isotopic. Moreover, the intersections $\mathcal{C}\cap B_{r_0}$
and $\wt{\mathcal{C}}\cap B_{r_0}$ are isotopic (it is crucial that $r_0$ be fixed before $\eta$). This shows that $\wt{\Theta}(r_0)=\Theta(r_0)$.
  
The link at infinity of $\wt{\mathcal{C}}$ is the torus link $T(\delta,\delta)$ and its diagram is the standard diagram of
$T(\delta,\delta)$: it has writhe $\delta(\delta-1)$, the rotation number $\delta$ and the Euler characteristic is $\rot-\writhe$.
This shows that $\wt{\Theta}(r)$ is zero for $r$ close to infinity.

\smallskip
\textbf{Step 3.} \emph{The conclusion.}

By Step 1 we obtain that $\wt{\Theta}(r)\equiv 0$. In Step 2 we showed 
that $\wt{\Theta}(r_0)=\Theta(r_0)$. This implies that $\Theta(r_0)=0$.
As $r_0$ was arbitrary, we infer that $\Theta(r)\equiv 0$.
\end{proof}

The following corollary concludes the proof of Theorem~\ref{main3} from the introduction.
\begin{corollary} \label{cor:ominus}
Only the attachment $\Ii^{\oplus}$ can occur (and not $\Ii^{\ominus}$).
\end{corollary}
\begin{proof}
The move $\Ii^{\ominus}$ decreases the value of $\Theta$. As $\Theta\equiv 0$,
the move $\Ii^{\ominus}$ cannot occur.
\end{proof}


\subsection{The self-linking of $L_r$ via the complex projection}\label{sec:bennequinsl}

Theorem~\ref{th:charlink} bears a strong resemblance to results of Bennequin and Laufer stated in Theorem~\ref{th:bennequin} below; see 
\cite[Proposition 4]{Ben} and \cite{Lau}.
In this section we shall show that there is no direct translation between Theorem~\ref{th:charlink} and the theorem of Bennequin--Laufer.
Moreover, combining the two results we can obtain subtle and non-trivial obstructions
for a link in $\Rr^3$ to be an image under the complex stereographic projection of a $\Cc$--link.
 
Let us recall results of \cite{Ben}.
For a point $z=(x_1,x_2,y_1,y_2)\in\Cc^2$ consider the vector $jz=(-y_1,y_2,x_1,-x_2)\in T_zS^3$.
That is, if we write $z=(x,y)$ as usual, we have $j(x,y)=(-\ol{y},\ol{x})$.
\begin{definition}\label{def:transandsl}\
\begin{itemize}
\item A link $L\subset S^3$ is transverse if it is transverse to the standard contact structure in $S^3$, that is for each $x\in L$, $T_xL\cap\ker\alpha=0$, where
$\alpha$ is a 1-form defining the standard contact structure.
\item Suppose $L\subset S^3$ is a transverse link. 
Define $jL$ to be the link $L$ pushed slightly along the vector field $jz$. The \defi{self-linking} number
$sl(L)$ is the linking number of $L$ and $jL$.
\end{itemize}
\end{definition}
It was shown in \cite{Ben} that $sl(L)$ is does not change if a link changes by a transverse isotopy (a transverse isotopy is an isotopy through transverse links). 
Moreover the following result holds.

\begin{theorem}[\expandafter{see \cite[Proposition 4]{Ben}}]\label{th:bennequin}
The self-linking number of a $\Cc$-link $L$ equals $-\chi(L)$.
\end{theorem}

By computing the image of the vector field under the map $\Psi_r$ we can try to  explicitly calculate self-linking numbers from
the diagram of complex stereographic projection. We have:
$$\dd\Psi_r(x,y)(j(x,y))
= \left(
-r\frac{r+\bar y}{(r+y)^2},\frac{-r}{|r+y|^2}\big(\Im(x)+\Im(\frac{y-\bar y}{r+y}\bar x)\big)
\right)$$

\begin{example}\label{ex:onL}
Consider the trivial knot $L_1= \{y=0\} \cap S_r$.
Its image under the complex stereographic
projection is $L_1' = \Psi_r(L_1) = \{ (e^{\ii \theta}, 0) \} \subset \Cc \times \Rr$.
For a point $(x,y)= (e^{\ii \theta}, 0)$ we get 
$\dd\Psi_r(x,y)(j(x,y)) = (-1,-\frac{\sin \theta}{r})$.
By shifting $L_1'$ along the vector field induced by $j$ we get a trivial knot $jL_1'$, which verifies
$\lk(L_1',jL_1') = -1$. Of course $L_1'$ bounds a disk and also verifies $\chi(L_1')=1$.
\end{example}

Bennequin's characterisation of self-linking numbers enables us to find explicit links that are not the image by $\Psi_r$ of any $\Cc$-link.
\begin{example}\label{ex:onLprim}
Let $0<\epsilon \ll 1$ and let $L_2' = \{ (\epsilon e^{\ii \theta} -\ii, 0) \}$ be a trivial knot of $\Cc \times \Rr$.
We will prove that $L_2'$ cannot be the image by $\Psi_r$ of any $\Cc$-link.
Suppose conversely that $L_2' = \psi_r(L_2)$ for some $\Cc$-link $L_2 \subset S_r$.
Since $L_2' \subset \Cc \times \{0\}$, then for $(x,y) \in L_2$ we get $y=0$, moreover as $L_2'$ is a small circle around
$(\ii,0) \in \Cc \times \{0\}$ then $x \sim -\ii r$.
It implies that $\dd\Psi_r(x,y)(j(x,y)) \sim (-1,+1) \in \Cc \times \Rr$. 
If $jL_2'$ is the push of $L_2'$ along $\dd\Psi_r(x,y)(j(x,y))$ then $jL_2'$ is not linked with $L_2'$ so that 
$\lk(L_2',jL_2') = 0$. 

If $L_2'$ were the projection of a $\Cc$-link, the Euler characteristic  $\chi(L_2')$ should be an odd number (using
the fact that $L_2'$ is strongly quasipositive and fibered, we could even calculate that $\chi(L_2')=1$). This contradicts Theorem~\ref{th:bennequin}.
\end{example}

Notice that the links $L_1$ from Example~\ref{ex:onL} and $L_2$ from Example~\ref{ex:onLprim} are transversely isotopic and
obviously 
$\writhe(L_1')=\writhe(L_2')$ and $\rot(L_1')=\rot(L_2')$. 
In particular Theorem~\ref{th:charlink} is not a consequence of Theorem~\ref{th:bennequin}.

\section{Diagrams of $\Cc$-links}\label{sec:diagramsandbraids}

\subsection{Yamada--Vogel algorithm}
\label{sec:vogelalgorithm}

The algorithm of Yamada--Vogel, see \cite{Yam,Vog}, can be effectively used to transform a 
link diagram into a braid. Let us quickly recall the algorithm. We will follow mainly \cite{KaTe}, 
another excellent source is \cite{BiBr}.

To begin with, notice that any two disjoint oriented circles in $S^2$ bound an annulus. We say that two oriented circles are \defi{compatible}
if they induce the same element in the homology group of the annulus.

Let $D\subset\Rr^2$ be a link diagram. Consider the diagram $D_{sm}$ obtained by smoothing each crossing of $D$ in an oriented way.
The diagram $D_{sm}$ is precisely the diagram used in the Seifert algorithm. It consists of a finite number of pairwise disjoint circles,
which are usually referred to as the \defi{Seifert circles}.
We denote by $n_+$ the number of positively (counter-clockwise) oriented circles and by $n_-$ the number of negatively (clockwise) oriented circles.
We also denote by $h(D)$ the number of pairs of incompatible Seifert circles. The Yamada--Vogel algorithm consists of two steps:
\begin{itemize}
\item[(1)] Apply $\reid_2$ and $\reid_3$ moves to $D$ in order to make $h(D)=0$. The resulting diagram $D_{sm}$ is a union of $n_+$
concentric positively oriented circles and $n_-$ concentric negatively oriented circles.
\item[(2)] Apply isotopy of the new diagram  in $S^2$, or, equivalently a sequence of Reidemeister moves (involving precisely $n_-$ moves 
$\reid_1^{\downarrow\boxminus\ominus}$ and $\reid_1^{\uparrow\boxminus\oplus}$) to obtain a diagram $D$ such that $D_{sm}$ is a union of concentric circles. It is clear that $D$ is then in a braid form.
\end{itemize}
If $n_-=0$, the second step is not necessary.

\subsection{Positive case}

We can now prove Theorem~\ref{main1} from the introduction.
\begin{theorem}
Suppose $L_r$ is a strongly quasipositive fibered link and its diagram $D_r$ has no negative Seifert circles (that is $n_-=0$). Then, after applying Yamada--Vogel's algorithm to $D_r$ we obtain a quasipositive braid.
\end{theorem}

\begin{proof}
Let $F$ be the fiber of the fibered link $L_r$. 
First, notice that as $\mathcal{C}$ is a smooth complex curve then by Kronheimer--Mrowka's theorem 
$\chi(\mathcal{C}_r) = \chi_4(L_r)$ (see \cite{BF}).
Second, as the link $L_r$ is fibered, $\chi(F) = \chi_3(L_r)$. Third, as the link $L_r$ is strongly quasipositive 
then $\chi_3(L_r) = \chi_4(L_r)$ (see \cite[Theorem 86]{Ru2005}).
By Theorem~\ref{th:charlink} we infer that 
$\chi(F)=\rot(D_r)-\writhe(D_r)$.

Apply the first step of Yamada--Vogel's algorithm and let $D'$ be the diagram obtained in that way. The Reidemeister moves involved in Step~1
of the algorithm are only $\reid_2$ and $\reid_3$, which preserve both the writhe and the winding number. In particular 
$\chi(F)=\rot(D')-\writhe(D')$.

By assumption $n_-=0$. In particular the second step is not necessary and $D'$ is already in braid form. The braid index of $D'$
is equal to $n_+=\rot(D')$. The algebraic length of $D'$ is equal to the writhe $\writhe(D')$.
We invoke now the following result of Etnyre and Van Horn-Morris. 

\begin{theorem}[see \expandafter{\cite[Theorem 5.4]{Et}}]
\label{th:et}
Let $L$ be a strongly quasipositive fibered link and let $B$ any braid representing it. Denote by $a(B)$ its algebraic length, $n(B)$
its braid index and $\chi(L)$ the Euler characteristic of a minimal genus Seifert surface.
Then $\chi(L)=n(B)-a(B)$ if and only if $B$ is a quasipositive braid.
\end{theorem}

In the context of this theorem, our equality $\chi(F)=\rot(D_r)-\writhe(D_r)$, is written $\chi(L)=n(B)-a(B)$.
Application of the `only if' part of Theorem \ref{th:et} concludes the proof of Theorem~\ref{main1}.
\end{proof}

\subsection{Removing negative circles}
The assumption that $n_-=0$ is quite restrictive. If it is not satisfied, we need to perform Step 2 of Yamada--Vogel's algorithm. In this
step we apply both the 
$\reid_1^{\downarrow\boxminus\ominus}$ and $\reid_1^{\uparrow\boxminus\oplus}$
move $n_-$ times. The first
of these moves does not affect the difference $\rot(D)-\writhe(D)$. However, the second one increases $\rot$ by $1$ (we create
a positive circle) and decreases the writhe by $1$. We conclude, that after Step 2, the resulting diagram $D'$ has
\[\chi(F)=\rot(D')-\writhe(D')-2n_-.\]
It represents a braid, but by the `if' part of Theorem~\ref{th:et},  $D'$ does 
not represent a quasipositive braid. In order to obtain a quasipositive braid, one needs to apply
a sequence of Markov moves. This is doable in each separate case, because of Markov's theorem, see \cite[Section 2.5]{KaTe}, 
but we do not have an explicit algorithm for that. In fact $D'$ is a braid representing a quasipositive link, that is to say there exists
a quasipositive braid $D''$ representing this link; Markov's theorem says that we can transform $D'$ to $D''$ by a sequence of Markov moves.

In some simple cases we can eliminate  negative circles by an explicit algorithm. By saying \defi{eliminate} we mean apply a sequence of
$\reid_2$, $\reid_3$, $\reid_1^{\boxplus\oplus}$ and $\reid_1^{\boxminus\ominus}$ moves on a diagram, so that the resulting
diagram has $n_-$ decreased by $1$. When all the negative circles are eliminated, we apply the Yamada--Vogel algorithm
and then Theorem~\ref{main1}. 

Before we investigate the possibility of eliminating negative Seifert circles, we need to introduce some terminology.
Suppose that $D\subset\Rr^2$ is a link diagram and let $D_{sm}$ be the smoothed diagram as in Section~\ref{sec:vogelalgorithm}.
\begin{definition}
The \defi{graph} $\Gamma_D$ of $D$ is a graph whose vertices are the Seifert circles of $D_{sm}$ and the edges correspond to crossings of $D$.
A crossing
of $D$ adjacent to two Seifert circles $C_1$ and $C_2$ 
corresponds to an edge connecting vertices $C_1$ and $C_2$.
Usually we will not distinguish between a Seifert circle and the vertex in $\Gamma_D$ corresponding to it.
Each vertex $C$ of $\Gamma_D$ has a sign, $\epsilon_C$, according to whether the Seifert circle is positively or negatively oriented.
Each edge $e$ of $\Gamma_D$ has also a sign. An edge corresponding to a positive crossing of $D$ has sign $\epsilon_e=+1$, an edge corresponding
to a negative crossing has sign $\epsilon_e=-1$.

The \defi{valency} of a Seifert circle $C$ is the valency of $C$ in the graph $\Gamma_D$, that is, the number
of incident edges. Two Seifert circles are \defi{adjacent}
if there is at least one edge connecting one to the other.
\end{definition}

\begin{remark}
There might be multiple edges connecting two vertices of the graph.
\end{remark}

The graph $\Gamma_D$ is \defi{bipartite}, that is, any closed path has an even number of edges. This is a consequence of the sign assignment to vertices of $\Gamma_D$. Notice that if two vertices of the same sign are connected by an edge, then the corresponding Seifert circles are nested; if the two vertices of the opposite sign are connected by an edge, then the corresponding Seifert circles
cannot be nested.

We begin with a simple result on removing a negative circle, later we pass to more complicated cases.

\begin{proposition}
\label{prop:neg1}
Suppose there is a negative circle with valency $1$ adjacent to a positive circle. 
Then it can be eliminated (in the sense described above). Moreover, if $D$ has $k$ negative circles with valency $1$ and each of these circles
is adjacent to a positive circle, then all of these circles can be removed.
\end{proposition}

\begin{proof}
By Corollary~\ref{cor:negative} a negative circle $C$ must have a crossing with negative jump. 
If the circle is adjacent to a positive circle, the crossing must be `outward' (relatively to $C$) as opposed to `inward'; see Figure~\ref{fig:singlecrossing}.
In theory, $C$ might contain other Seifert circles inside it, but
we can move them out of $C$ by using $\reid_2$ and $\reid_3$ moves. Then $C$ has a single negative crossing, so it can be removed
by a $\reid_1^{\downarrow\boxminus\ominus}$ move.
\end{proof}

\begin{figure}
\input{pic-link-29.tex}
\caption{Two negative circles of valency one with negative crossing (type $\ominus/\boxminus$). 
To the left: an `outward' crossing. To the right: an `inward' crossing.}
\label{fig:singlecrossing}
\end{figure}
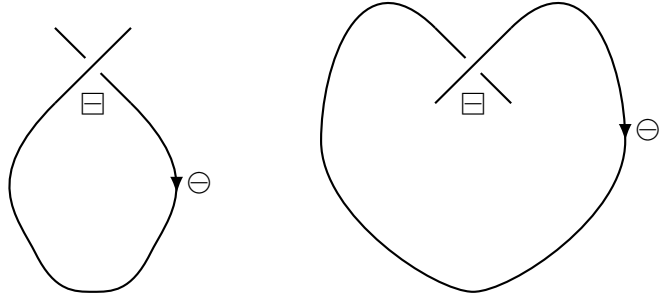

\begin{remark}
We point out that the proof of Proposition~\ref{prop:neg1} involves the use of contact transversality of the link in the sense of Section~\ref{sec:contact2}.
\end{remark}

In the next case we discuss a situation, when a circle has valency two.

\begin{proposition}
\label{prop:valencytwo}
Suppose $D$ is a diagram of a complex projection of some $\Cc$--link. Assume that $D$ has a negative Seifert circle $C$ with
valency $2$, which is adjacent to two different circles, which are both positively oriented. Then one can eliminate the Seifert circle $C$.
\end{proposition}

\begin{proof}
If a negative circle is adjacent to two positive circles, none of these two circles can be nested inside $C$. Given that $C$ is  negatively
oriented, it must have at least one negative crossing (Corollary \ref{cor:negative}). This leaves us with two possibilities for the position of Seifert circles. We sketch them in Figure~\ref{fig:possibilities}.

The first case is solved by a single $\reid_2$ move. 
In the second case we use the trick explained in Figure~\ref{fig:newsituation}.
We perform a $\reid_1^{\uparrow\boxplus\oplus}$ move followed by a series of $\reid_2$ and $\reid_3$
moves. After theses moves the central negative Seifert circle has disappeared.
A new positive Seifert circle is created.
If $C$ was adjacent to two different positively oriented Seifert circles,
then these two positive circles will form a single positively oriented Seifert circle 
after the moves.
\end{proof}

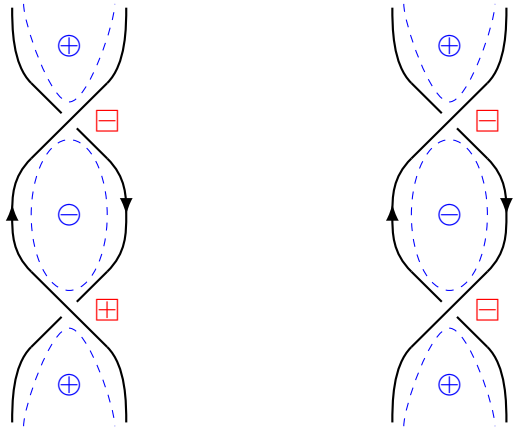
\begin{figure}
\input{pic-link-30.tex}
\caption{The two possibilities for the position of Seifert circles in Proposition~\ref{prop:valencytwo}.}\label{fig:possibilities}
\end{figure}

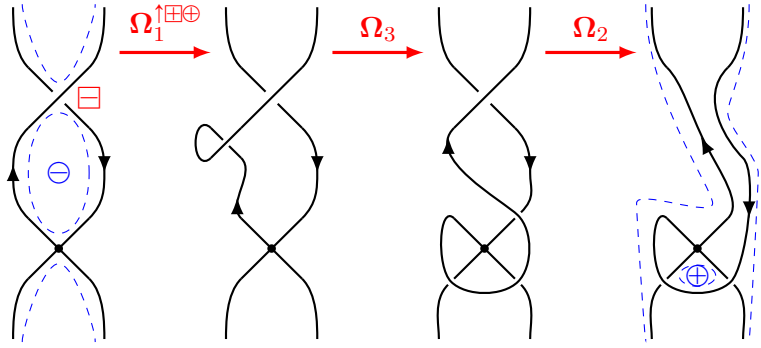
\begin{figure}
\input{pic-link-31.tex}
\caption{Trick to remove a negative Seifert circle of valency $2$ 
(the dot crossing can be replaced by any type $\boxplus$/$\boxminus$ of crossing).}\label{fig:newsituation}
\end{figure}

\begin{remark}
\label{rem:twonegative}
If we have two negative Seifert circles of valency $2$ and each is adjacent to two disjoint positive circles, then we cannot in general remove the
two negative Seifert circles at once. Indeed the two negative circles might be adjacent to the same two positive circles. The first application
of the algorithm in Proposition~\ref{prop:valencytwo} will change these two positive circles into a single positive circle. 
Thus, it will not be possible to apply the algorithm again to remove the other negative circle.
\end{remark}

\subsection{Index of the graph and negative circles}
\label{sec:murasugiprzytycki}

The problem described in Remark~\ref{rem:twonegative} can be approached using the so-called graph index. Indeed, the algorithm presented
in Proposition~\ref{prop:valencytwo} can be thought of as a simple variant of the Murasugi--Przytycki move from \cite{MP}; even though the
Murasugi--Przytycki move applied to the situation on the left of Figure~\ref{fig:newsituation} gives a different output, the
philosophy remains the same. Our approach
will be based on the recent paper of Traczyk, see \cite{Tr2}, however we will focus more on the type of the $\reid_1$ moves
used in the Murasugi--Przytycki move. Therefore, we will need yet another variant of an index of the graph.

The Murasugi--Przytycki move on a link diagram $D$ is the move depicted in Figure~\ref{fig:stoimenow}. 
The procedure is as follows. We regard the diagram $D$ as Seifert circles joined with bands (this is the way one looks at the diagram when one constructs
the Seifert surface from the diagram).
Take two Seifert circles $C_1$ and $C_2$ and suppose they are connected by a single crossing. Take the bottom strand (the tunnel)
of the crossing and replace it by an arc constructed as follows: start shortly before the crossing (on the side of $C_2$), make a U-turn and follow the
Seifert circle $C_1$. The arc goes along Seifert circles according to the following procedure: if a band is met, then the band either connects $C_1$
with another Seifert circle, or it connects some two other Seifert circles. In the first case, the arc goes along the band and then along the new Seifert circle.
In the second case, the arc goes underneath the band and still follows the same Seifert circle it was following before meeting the band. We refer
to \cite{St2,Tr2} and to \cite[Section 8]{MP} for more details.
We have the following observation.

\begin{figure}
\input{stoimenow.tex}
\caption{A move of Murasugi--Przytycki. The short (red dashed) strand, is replaced by the long (green dotted) arc, that goes
below all the other crossings. Picture taken from \cite{St2}.}
\label{fig:stoimenow}
\end{figure}
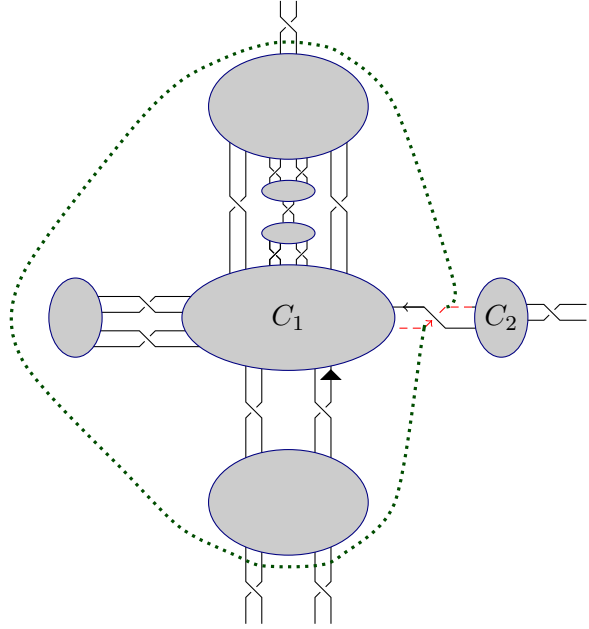

%
%
\begin{lemma}\ 
\label{lem:MPproperties}
\begin{itemize}
\item The effect of a Murasugi--Przytycki move is that the Seifert circles $C_1$ and $C_2$ are merged into one Seifert circle whose sign is the
same as the sign of $C_2$. All the other Seifert circles are preserved. The precise effect
on the graph $\Gamma_D$ is a variant of a graph contraction followed by a one point sum with another graph. We refer the reader to  \cite[Section 5]{St2}
for a precise description.
\item The Murasugi--Przytycki move is made using a sequence of $\reid_2$ and $\reid_3$ moves together with a single $\reid_1^{\downarrow}$ move.
The signs $\oplus/\ominus$, $\boxplus/\boxminus$ of the $\reid_1$ move depend on the orientation of the circle $C_1$ and the sign of the crossing. 
For instance, if $C_1$ is a negative Seifert circle connected to $C_2$ by a negative crossing, then the move is 
$\reid_1^{\downarrow\boxminus\ominus}$.
\end{itemize}
\end{lemma}

\begin{proof}
The first part is proved in \cite{MP,St2}. 
The proof of the second part consists of presenting the move as adding a small loop (that is the $\reid_1$ move) on the dashed part
of Figure~\ref{fig:stoimenow} and then moving this loop under the circles adjacent to $C_1$ to obtain the dotted curve. The latter
operation is a sequence of $\reid_2$ and $\reid_3$ moves.
The details are left to the reader.
\end{proof}

We will now define an important notion of this subsection.
\begin{definition}\label{def:hashindex}
Let $\Gamma$ be a bipartite diagram with signs attached to vertices and edges. The \emph{doubly negative index} $\ind_{\#}\Gamma$
is the maximal number of  edges $e_1,\ldots,e_n$, such that
\begin{itemize}
\item each of the $e_i$ is a negative edge;
\item each of the $e_i$ is adjacent to at least one negative vertex;
\item the edges are \emph{cyclically} independent, that is, no $k$ edges lie on a cycle of length $2k$ or less.
\end{itemize}
\end{definition}
\begin{remark}
Notice that being cyclically independent for a single edge $e_1$ means precisely that this is the only edge connecting the two circles.
The potential problem with reducing negative Seifert circles in Remark~\ref{rem:twonegative} is due to lack of cyclical independence of a pair
of edges.
\end{remark}

We can now prove the result, which clearly implies Theorem~\ref{main2} from the introduction.
We have the following result.
\begin{theorem}
\begin{itemize}
\item[(a)] Suppose that $D$ is a link diagram having $n_-$ negatively oriented circles. Then at least $\ind_{\#}(\Gamma_D)$ circles can
be eliminated. 

\item[(b)] In particular, if $D$ is a diagram of a strongly quasipositive fibered $\Cc$-link and $n_-=\ind_{\#}(\Gamma_D)$, then the Yamada--Vogel algorithm combined with
the Murasugi--Przytycki algorithm make $D$ into a quasipositive braid.
\end{itemize}
\end{theorem}

\begin{proof}
We proceed by induction. Suppose $\ind_{\#}\Gamma=n$ and let $e_1,\ldots,e_n$ be the edges satisfying the three points 
of Definition~\ref{def:hashindex}.
Suppose $e_1$ connects Seifert circles $C_1$ and $C_2$, with $C_1$ negative. Apply the Murasugi--Przytycki move. 
According to Lemma~\ref{lem:MPproperties}, this move decreases $n_-$ by $1$. Call the new diagram $D_1$.
It has $n_--1$ negative circles (because $C_1$ has disappeared). We argue that the index of $\Gamma_{D_1}$ is at least $\ind_{\#}(\Gamma_D)-1$.

To show this, notice that by \cite[Proof of Theorem 3]{Tr1} the set of cyclically independent edges is also independent. 
The definition of independence is quite
intricate, but the bottom line is that after applying the Murasugi--Przytycki move along the edge $e_1$, the remaining edges $e_2,\ldots,e_n$
are independent and 
cyclically independent in the new graph $D_1$; see \cite{MP,St2} for a detailed description of the effect of the Murasugi--Przytycki operation
on graphs and \cite{MP,Tr1} for the relation between the graph index and the cyclic index. In particular $\ind_{\#}(\Gamma_{D_1})\ge \ind_{\#}(\Gamma_D)-1$.

We apply the Murasugi--Przytycki move again, this time to $e_2$. Repeating the process we eliminate $\ind_\#(\Gamma_D)$ negative circles.

If $\ind_\#=n_-$ then the procedure eliminates all the negative circles and we conclude by Theorem~\ref{main1}.
\end{proof}

Note that unlike in the proof of Proposition~\ref{prop:neg1},
in the proof of the first part of Theorem~\ref{main2} we did not use the fact that $D$ is a diagram of a transverse link. Therefore there might
be a possibility to have some more control over the quantity $\ind_{\#}(\Gamma_D)$. We conclude this section by two open questions.
\begin{question}\
\begin{itemize}
\item Is it true that for a diagram of a transverse link $\ind_{\#}(D)=\ind_-(D)$?
\item Is it true that for a diagram of a $\Cc$-link we have $\ind_{\#}(D)=n_-$?
\end{itemize}
\end{question}

\section{Proofs}\label{sec:proofs}

\subsection{Preliminaries}\label{sec:preliminaries}

We work in local coordinates. Write $z_0=(x_0,y_0) \in \mathcal{C} \subset \Cc^2$. Suppose $z_0$
is a smooth point of $\mathcal{C}$
and consider the Taylor expansion of $f$ near $z_0$.
\begin{equation}\label{eq:Ftaylor}
f(x_0+\eta,y_0+\xi)=a\eta+b\xi+c\eta^2+d\eta\xi+e\xi^2+O(\|(\eta,\xi)\|^3)
\end{equation}
for some complex numbers $a,b,c,d$ and $e$. Choose a 
local parametrization of $\mathcal{C}$ near $z_0$:
\begin{equation}\label{eq:local}
t\mapsto z_0+t(\alpha,\beta)+t^2(\gamma,\delta) + O(t^3)
= \left\{\begin{array}{rcl}
x(t) &=& x_0 + \alpha t + \gamma t^2 + O(t^3)\\
y(t) &=& y_0 + \beta t + \delta t^2 + O(t^3)
\end{array}\right.
\end{equation}
for some complex numbers $\alpha,\beta,\gamma$ and $\delta$. We can find 
relations between $a,b,c,d,e$ and $\alpha,\beta,\gamma,\delta$ by substituting
\eqref{eq:local} into \eqref{eq:Ftaylor}. We immediately recover that $(a,b)$ 
is related to $(\alpha,\beta)$ by the relation $a\alpha+b\beta=0$. 
As $\grad F(z_0)=(a,b)\neq (0,0)$, changing $t$ by a complex factor, we can and will actually assume that 
\[a=-\beta \quad \text{ and } \quad b=\alpha.\] 
The constants $\gamma$ and $\delta$ are related to the coefficients 
$c,d,e$ by the formula
\begin{equation}\label{eq:seconddiff}
a\gamma + b\delta+c\alpha^2+d\alpha\beta+e\beta^2=0.
\end{equation}

The function ``square of the distance to origin'' on $\mathcal{C}$ is given in the local 
parametrization \eqref{eq:local} by
\begin{equation}
\label{eq:expandt}
\begin{split}
t\mapsto \|z_0\|^2+&2\Re \big[ t(\ol{x_0}\alpha+\ol{y_0}\beta) \big]+\\
&+|t|^2(|\alpha|^2+|\beta|^2)+2\Re 
\big[t^2(\ol{x_0}\gamma+\ol{y_0}\delta)\big]+O(|t|^3).
\end{split}
\end{equation}


Let 
\begin{equation}\label{eq:hklnew}
h=|\alpha|^2+|\beta|^2 \qquad k=\Re(\ol{x_0}\gamma+\ol{y_0}\delta) \qquad 
\ell=\Im(\ol{x_0}\gamma+\ol{y_0}\delta).
\end{equation}
After writing $t=u+\ii v$, the square of the distance function $t = (u,v) \mapsto 
|x(t)|^2+|y(t)|^2$ has the following Hessian:
\[H = \begin{pmatrix} h+2k && -2\ell\\-2\ell && h-2k \end{pmatrix}.\]
Since $h>0$, the trace of this matrix is positive, so that the index of the 
critical point is either $0$ or $1$. 

The determinant is \[\det H = h^2-4k^2-4\ell^2.\] 
\begin{itemize}
  \item if $\det H > 0$, this is the \defi{elliptic case} $\mathcal{E}$;
  \item if $\det H < 0$, this is the \defi{hyperbolic case} $\mathcal{H}$;
  \item if $\det H = 0$, this is the \defi{degenerate case}.
\end{itemize}

The degenerate case can be excluded at points of interest by imposing genericity conditions on $\mathcal{C}$; see Section~\ref{sec:afewwords}
for more details.

\subsection{Proof of Proposition~\ref{prop:localindex}}\label{sec:proofoflocalindex}

The first result links the nature of a non-transversality point 
with the topological modification at this point.
Let us recall the statement of Proposition~\ref{prop:localindex}.

\begin{proposition}
The intersection index of $\mathcal{C}$ and $\{Jf(x,y)=0\}$ at $z$ is $1$ 
if and only if $z$ is the saddle point (of Morse index $1$, case $\mathcal{H}$) of the 
intersection. 
It corresponds to move $\Ii$, the attachment of a $1$-handle.
  
If the intersection index is $-1$, then $z$ 
is a local minimum (of Morse index $0$, case $\mathcal{E}$). 
It corresponds to move $\Oo$, the birth of a component.
\end{proposition}

\begin{remark}
It will follow from the local description of the handle attachment given below that the
0-handle corresponds to adding to the link diagram a round circle and not a more complicated
diagram of an unknot. Likewise, a 1-handle corresponds to adding a handle as
in Figure~\ref{fig:iiiplus}, that is, the 1-handle is not tangled in any way.

The reason why more complicated changes do not occur when the handle is attached is
the genericity condition, more precisely, Condition~\ref{con:almosttrivial}. 
\end{remark}

\begin{lemma}
\label{lem:handleindex}
The critical point has Morse index $0$ or $1$.
The Morse index is $1$ if and only if 
$\left(|\alpha|^2+|\beta|^2\right)^2<4|\beta\gamma-\alpha\delta|^2$.
\end{lemma}

\begin{proof}
The condition for $(x_0,y_0)$ to be a critical point of the square of 
the distance to origin function (i.e. $Jf(x_0,y_0)=0$) means that 
$\langle (x_0,y_0) \mid (\alpha,\beta) \rangle_\Cc=0$. 
On rescaling $f$ and $(x,y)$ we can actually suppose that 
\begin{equation}
\label{eq:z1z2}
\ol{x_0}=\beta \quad \text{ and } \quad \ol{y_0}=-\alpha.
\end{equation}
The first order term of \eqref{eq:expandt} vanishes.
The second order term of \eqref{eq:expandt} takes the form
\[|t|^2(|\alpha|^2+|\beta|^2)+2\Re \big[ t^2(\beta\gamma-\alpha\delta)\big].\]
In this situation we get :
\begin{equation}\label{eq:hkl}
h=|\alpha|^2+|\beta|^2 \qquad k=\Re(\beta\gamma-\alpha\delta) \qquad 
\ell=\Im(\beta\gamma-\alpha\delta).
\end{equation}
As above the Hessian
\[H = \begin{pmatrix} h+2k && -2\ell\\-2\ell && h-2k \end{pmatrix}.\]
has index $0$ or $1$. 
As $\det H = h^2-4k^2-4\ell^2$, we obtain the result.
\end{proof}

\bigskip

In conjunction with the next lemma we will get the proof of Proposition 
\ref{prop:localindex}.
\begin{lemma}
\label{lem:intersectionindex}
The intersection index of $\mathcal{C}$ and $\{Jf(x,y)=0\}$ is $-1$ or $+1$.
The intersection index is $+1$ if and only if 
$\left(|\alpha|^2+|\beta|^2\right)^2<4|\beta\gamma-\alpha\delta|^2$.
\end{lemma}

\begin{proof}
For the proof of this lemma we substitute \eqref{eq:Ftaylor} into \eqref{eq:Jf} 
(with $x=x_0+\eta$ and $y=y_0+\xi$): 
\[Jf(x,y)=\left|\begin{matrix} -\beta+2c\eta+d\xi &\alpha+d\eta+2e\xi \\
\ol{x_0+\eta} & \ol {y_0+\xi}\end{matrix}\right|+O(\|\eta,\xi\|^2).\]
The linear terms in $\eta$ and $\xi$ of $Jf(x,y)$ are
\begin{align*}
LJf(x,y)&=-\beta\ol{\xi}-(2c\eta+d\xi)\alpha 
-(d\eta+2e\xi)\beta-\alpha\ol{\eta}\\
&=\eta(-\beta d-2c\alpha)+\ol{\eta}(-\alpha)+\xi(-\alpha d-2\beta 
e)+\ol{\xi}(-\beta),
\end{align*}
where we substituted $\ol{x_0}=\beta$ and $\ol{y_0}=-\alpha$ as in 
\eqref{eq:z1z2}.

\smallskip
Suppose $\{Jf(x,y)=0\}$ intersects (smoothly) transversally with $\mathcal{C}$ at $z_0$. 
Then the intersection index is equal to the 
intersection index of the linearized equations, that is of
\[\{LJf(x,y)(\eta,\xi)=0\}\quad\text{ and }\quad\{a\eta+b\xi=0\}.\]
A parameterization of $\{a\eta+b\xi=0\}$ is given by $t\mapsto (\alpha t,\beta 
t)$.
The intersection index is $+1$ or $-1$ depending on whether the map from $\Rr^2$ 
to $\Rr^2$ given by
\[t\mapsto LJf(x,y)(\alpha t,\beta t)\]
preserves or changes the orientation. Explicitly this map is given by
\begin{equation}\label{eq:restrictedmap}
t\mapsto \left[-\alpha(\beta d+2c\alpha)-\beta(\alpha d+2\beta e)\right] t
+(-|\alpha|^2-|\beta|^2)\ol{t}.
\end{equation}
Notice that the expression in brackets is by \eqref{eq:seconddiff} equal to 
$2(-\beta\gamma+\alpha\delta)$.
With the notation of \eqref{eq:hkl} we rewrite this as
\[t\mapsto -2(k+\ii\ell)t-h\ol{t},\]
that is, in real coordinates $(u,v)$ such that $t=u+\ii v$,
\[(u,v)\mapsto-\begin{pmatrix} 2k+h & -2\ell \\ 2\ell & 
2k-h\end{pmatrix}\begin{pmatrix} u \\ v\end{pmatrix}.\]
The map preserves the orientation if and only if $4(k^2+\ell^2)>h^2$. 
\end{proof}

\bigskip
Combining Lemma~\ref{lem:intersectionindex} with Lemma~\ref{lem:handleindex}, we see that the intersection 
index of 
$\mathcal{C}$ and $\{Jf(x,y)=0\}$ at $z_0$ is positive if and only if it 
corresponds to a $1$-handle and is negative if and only if it corresponds to a $0$-handle.

\subsection{Proof of Lemma~\ref{lem:critproj}}\label{sec:proofofcritproj}

We recall the statement of Lemma~\ref{lem:critproj}.
\begin{lemma}
The critical points of the projection $\pi_r : L_r \to \Rr^2$
defined by $\pi_r = \pi'\circ \Psi_r$ are:
\begin{itemize}
  \item the points of non-transversality of $\mathcal{C}$ with $S_r$ given by the 
equation 
  $\begin{vmatrix}
\frac{\partial f}{\partial x}  & \frac{\partial f}{\partial y} \\
\ol x & \ol y \\
\end{vmatrix} = 0$;

  \item the points verifying 
  $x \frac{\partial f}{\partial x} + (r+y)\frac{\partial f}{\partial y} = 0$.
\end{itemize}
\end{lemma}

\begin{proof}

The proof is given in four steps. First we formulate the statement in real 
coordinates. Then we investigate critical points of the function given in real 
coordinates.

\medskip
\textbf{Step 1.} \emph{Complex vs real.}

\smallskip
The map $f : (x,y) \mapsto f(x,y)$ is written
in real coordinates in the following way.
$$f : (x, \ol x, y, \ol y) \mapsto \big(\Re(f), \Im(f) \big) 
= \left( \frac{f+\ol f}{2} , \frac{f-\ol f}{2\ii} \right)$$

We define $\rho_r(x,y) = |x|^2+|y|^2 - r^2$, so that $S_r = \{\rho_r=0\}$.
Be definition of $\pi_r = \pi' \circ \Psi_r$, 
we have $\pi_r : S_r \to \Cc$ where $S_r \subset \Cc^2$ with
$\pi_r(x,y) = \frac{x}{r+y}$, and seen as a real map 
$\pi_r : S_r \to \Rr^2$ where $S_r \subset \Rr^4$:
\begin{eqnarray*}
\pi_r(x,\ol x, y, \ol y) 
 & = & \left(\Re\left(\frac{x}{r+y}\right), \Im\left(\frac{x}{r+y}\right) 
\right) \\
 & = & \left( \frac{1}{2}\left(\frac{x}{r+y} + \frac{\ol x}{r+\ol y}\right) , 
\frac{1}{2\ii}\left(\frac{x}{r+y} - \frac{\ol x}{r+\ol y}\right) \right).  
\end{eqnarray*}

As $L_r$ is real one dimensional, the critical point of $\pi_r$ are 
the points $(x,y)$ such that $\dd \pi_r (x,y)=(0,0)$.
Equivalently these are the points $(x,y)$ such that
$\dd \Re(\pi_r) (x,y) = 0$ and $\dd \Im(\pi_r) (x,y) = 0$.

\bigskip

\textbf{Step 2.} \emph{Critical points of the real part.}

\smallskip
The critical points of $\Re(\pi_r)$ are obtained by considering the critical 
points
of $\Re(\pi_r)$ restricted to the set $\{\Re(f)=0\} \cap \{\Im(f)  = 0\} \cap \{\rho_r 
= 0\}$.
That is, the critical points are given by equation:
\begin{equation*}
\begin{vmatrix}
\frac{\partial \Re f}{\partial x} & \frac{\partial \Re f}{\partial \ol x} &
\frac{\partial \Re f}{\partial y} & \frac{\partial \Re f}{\partial \ol y} \\
\frac{\partial \Im f}{\partial x} & \frac{\partial \Im f}{\partial \ol x} &
\frac{\partial \Im f}{\partial y} & \frac{\partial \Im f}{\partial \ol y} \\
\frac{\partial \rho_r}{\partial x} & \frac{\partial \rho_r}{\partial \ol x} &
\frac{\partial \rho_r}{\partial y} & \frac{\partial \rho_r}{\partial \ol y} \\
\frac{\partial \Re \pi_r}{\partial x} & \frac{\partial \Re \pi_r}{\partial \ol x} &
\frac{\partial \Re \pi_r}{\partial y} & \frac{\partial \Re \pi_r}{\partial \ol y} 
\\
\end{vmatrix} = 0  
\end{equation*}

Remember that $\Re f = \frac{1}{2}(f+\ol f)$, $\Im f = \frac{1}{2\ii}(f-\ol 
f)$.
Since $f$ is holomorphic, $\frac{\partial f}{\partial \ol x} = 0$,
$\frac{\partial \ol f}{\partial x} = 0$, and 
$\overline{\frac{\partial f}{\partial x}} = \frac{\partial \ol f}{\partial \ol 
x}$. 
Also $\rho_r = |x|^2+|y|^2 -r^2= x\ol x+ y\ol y -r^2$,
$\Re(\pi_r) = \frac{1}{2}\left(\frac{x}{r+y} + \frac{\ol x}{r+\ol y}\right)$.

This yields: 
\begin{equation*}
\begin{vmatrix}
\frac{\partial f}{\partial x} & \overline{\frac{\partial f}{\partial x}} &
\frac{\partial f}{\partial y} & \overline{\frac{\partial f}{\partial y}} \\
\frac{\partial f}{\partial x} & -\overline{\frac{\partial f}{\partial x}} &
\frac{\partial f}{\partial y} & -\overline{\frac{\partial f}{\partial y}} \\
\ol x & x & \ol y & y \\
\frac{1}{r+y} & \frac{1}{r+\ol y} & -\frac{x}{(r+y)^2} & -\frac{\ol x}{(r+\ol 
y)^2} \\
\end{vmatrix} = 0 
\quad \text{i.e.} \quad
\begin{vmatrix}
\frac{\partial f}{\partial x} & \overline{\frac{\partial f}{\partial x}} &
\frac{\partial f}{\partial y} & \overline{\frac{\partial f}{\partial y}} \\
\frac{\partial f}{\partial x} & 0 &
\frac{\partial f}{\partial y} & 0 \\
\ol x & x & \ol y & y \\
\frac{1}{r+y} & \frac{1}{r+\ol y} & -\frac{x}{(r+y)^2} & -\frac{\ol x}{(r+\ol 
y)^2} \\
\end{vmatrix} = 0 
\end{equation*}
We expand the last determinant along the last line.

\begin{equation*}
\frac{-1}{r+y} 
\begin{vmatrix}
\overline{\frac{\partial f}{\partial x}} & \frac{\partial f}{\partial y} & 
\overline{\frac{\partial f}{\partial y}} \\
 0 & \frac{\partial f}{\partial y} & 0 \\
x & \ol y & y \\
\end{vmatrix}  
+\frac{1}{r+\ol y}
\begin{vmatrix}
\frac{\partial f}{\partial x} & \frac{\partial f}{\partial y} & 
\overline{\frac{\partial f}{\partial y}} \\
\frac{\partial f}{\partial x} & \frac{\partial f}{\partial y} & 0 \\
\ol x &  \ol y & y \\
\end{vmatrix}
+\frac{x}{(r+y)^2}
\begin{vmatrix}
\frac{\partial f}{\partial x} & \overline{\frac{\partial f}{\partial x}} & 
\overline{\frac{\partial f}{\partial y}} \\
\frac{\partial f}{\partial x} & 0  & 0 \\
\ol x & x  & y \\
\end{vmatrix}
-\frac{\ol x}{(r + \ol y)^2}
\begin{vmatrix}
\frac{\partial f}{\partial x} & \overline{\frac{\partial f}{\partial x}} 
&\frac{\partial f}{\partial y}  \\
\frac{\partial f}{\partial x} & 0 &\frac{\partial f}{\partial y}  \\
\ol x & x & \ol y \\ 
\end{vmatrix}=0
\end{equation*}
Denote 
$$J= \begin{vmatrix}
\frac{\partial f}{\partial x} & \frac{\partial f}{\partial y}  \\
\ol x &  \ol y \\ 
\end{vmatrix}.$$
The critical points are given be the equation:
$$
-\frac{1}{r+y}\frac{\partial f}{\partial y}\overline{J}
+\frac{1}{r+\ol y}\overline{\frac{\partial f}{\partial y}}J
-\frac{x}{(r+y)^2}\frac{\partial f}{\partial x}\overline{J}
+\frac{\ol x}{(r+\ol y)^2}\overline{\frac{\partial f}{\partial x}}J,$$
which is equivalent to the real equation:
\begin{equation}
\Im\left( 
\left(\frac{1}{r+\ol y}\overline{\frac{\partial f}{\partial y}} + 
\frac{\ol x}{(r+\ol y)^2}\overline{\frac{\partial f}{\partial x}}\right)
\begin{vmatrix}
\frac{\partial f}{\partial x}  & \frac{\partial f}{\partial y} \\
\ol x & \ol y \\
\end{vmatrix}
\right) = 0.
\end{equation}

\bigskip

\textbf{Step 3.} \emph{Critical points of the imaginary part.}

\smallskip
Acting as in Step 2 we show that the critical points of $\Im(\pi_r)$ are the points of $L_r$ 
given by the real equation:
\begin{equation}
\Re\left( 
\left(\frac{1}{r+\ol y}\overline{\frac{\partial f}{\partial y}} + 
\frac{\ol x}{(r+\ol y)^2}\overline{\frac{\partial f}{\partial x}}\right)
\begin{vmatrix}
\frac{\partial f}{\partial x}  & \frac{\partial f}{\partial y} \\
\ol x & \ol y \\
\end{vmatrix}
\right) = 0.
\end{equation}

\bigskip

\textbf{Step 4.} \emph{Conclusion of the proof.}

\smallskip
The critical points of $\pi_r$ are the points of $L_r$ given by the equation:
\begin{equation}
 \left(\frac{1}{r+y}\frac{\partial f}{\partial y} + 
\frac{x}{(r+y)^2}\frac{\partial f}{\partial x}\right)
\begin{vmatrix}
\frac{\partial f}{\partial x}  & \frac{\partial f}{\partial y} \\
\ol x & \ol y \\
\end{vmatrix}
 = 0.
\end{equation}

\end{proof}
\subsection{Proof of Theorem~\ref{th:reid1}}\label{sec:proofofreid1}
For the reader's convenience we recall the statement of Theorem~\ref{th:reid1}

\begin{theorem}
Suppose that $(x_0,y_0) \in L_{r_0}$ is a critical point of the projection  
$\pi_{r_0}$, but $(x_0,y_0)$ is a point where
$\mathcal{C}$ and $S_{r_0}$ have (smoothly) transverse intersection.
Then for $r$ near $r_0$ (going from $r_0-\epsilon$ to $r_0+\epsilon$, 
$\epsilon>0$) 
the topological type of the link is unchanged but
the projection $D_r = \pi_r(L_r) \subset \Rr^2$ changes by a Reidemeister
move of type $\reid_1^{\uparrow\boxplus\oplus}$ or 
$\reid_1^{\downarrow\boxminus\ominus}$.
\end{theorem}
\begin{proof}

\textbf{Step 1.} \emph{Local parameterization of $\mathcal{C}$.}

\smallskip

Recall that a local parameterization of 
the curve $\mathcal{C}$ near 
$(x_0,y_0)$ is given by 
\begin{equation}\label{eq:localparametr}
\left\{\begin{array}{rcl}
x(t) &=& x_0 + \alpha t + \gamma t^2 + O(t^3)\\
y(t) &=& y_0 + \beta t + \delta t^2 + O(t^3)
\end{array}\right.
\end{equation}

Denote by $r_0$ the radius of the sphere through $(x_0,y_0)$:
$r_0^2 = |x_0|^2+|y_0|^2$. 
By Lemma \ref{lem:critproj}, as 
$(x_0,y_0)$ is a critical point of $\pi_{r_0}$, we have 
\[x_0 \frac{\partial f}{\partial x}(x_0,y_0) + (r_0+y_0)\frac{\partial 
f}{\partial y}(x_0,y_0) = 0.\]

\bigskip

\textbf{Step 2.} \emph{Local equation of the parameters of the link $L_r$.}

\smallskip
Given a local parametrization of $\mathcal{C}$ we calculate the set of parameters
that get mapped to the sphere $S_r$.

As $L_r = \mathcal{C} \cap S_r$, combining \eqref{eq:localparametr} with the equation
of $S_r$ given by $\{|x|^2+|y|^2=r^2\}$ we obtain the following condition for $t$
(see also \eqref{eq:expandt}):
$$2\Re \big[ t(\ol{x_0}\alpha+\ol{y_0}\beta) \big]
+|t|^2(|\alpha|^2+|\beta|^2)
+2\Re\big[t^2(\ol{x_0}\gamma+\ol{y_0}\delta)\big]
+O(|t|^3) = r^2-r_0^2$$

We denote $t =  u + \ii v$, $\omega = \ol{x_0}\alpha+\ol{y_0}\beta$ and use 
the notations of \eqref{eq:hklnew} 
to get a local equation for the $t$, satisfying $(x(t),y(t)) \in L_r$:
$$2\Re \big[ (u+\ii v)\omega \big] 
+ (h+2k)u^2 + (h-2k)v^2 -4\ell uv + O(\|(u,v)\|^3) = r^2-r_0^2$$


For  $\omega,h,k,\ell$ fixed, it is a family of real curves in $\Rr^2$, depending on $r$.
Suppose the Hessian $H$ is non-degenerate.
Keeping only the Taylor expansion of order $2$,
we get a family of conics whose center is fixed.
These conics are given by equation
\begin{equation}\label{eq:conics}
2\Re \big[ (u+\ii v)\omega \big] 
+ (h+2k)u^2 + (h-2k)v^2 -4\ell uv = r^2-r_0^2.
\end{equation} 
These conics are tangent to lines perpendicular to the real line 
$(0\ol\omega)$.
The point of tangency for the conics corresponding to $r>r_0$
is on the half line $[0,\ol\omega)$
(while for $r<r_0$ it is on the half-line $[0,-\ol\omega)$); see Figure~\ref{fig:localtangents}.

  \begin{figure}[h]
    \input{pic-link-13.tex}
    \caption{Local tangents of the parameterizations.}\label{fig:localtangents}
  \end{figure}
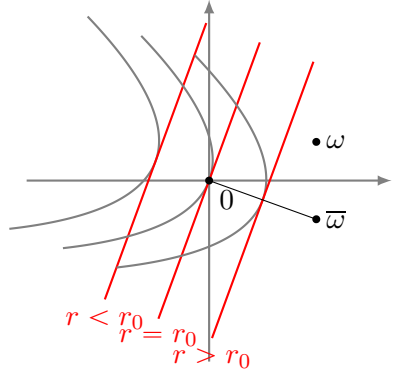

The family of these conics is:
\begin{itemize}
  \item in case $\mathcal{E}$: a family of ellipses,
  \item in case $\mathcal{H}$: a family of hyperbolas.
  \end{itemize}
This justifies the terminology defined in the last paragraph of Section~\ref{sec:preliminaries}.

It is crucial to note that the local pictures for the two situations are similar, but when $r$ goes from
$r<r_0$ to $r>r_0$ the movie is reversed. The key point will be that  
ellipses that intersect the tangent at the origin in two points are for radius $r>r_0$, while hyperbolas that intersect the tangent at the origin in two points are for radius $r<r_0$.

  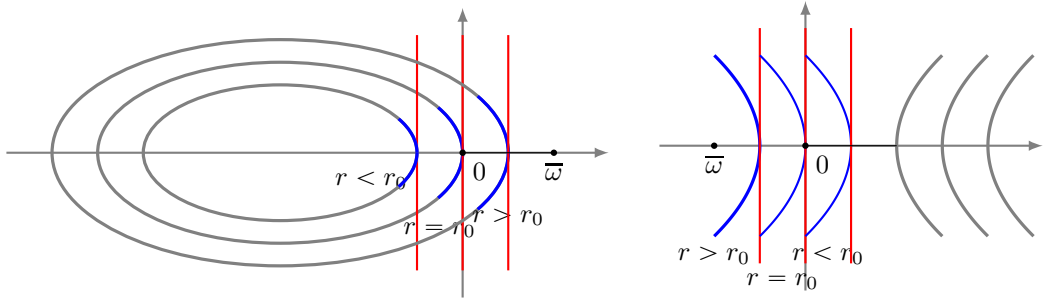
\begin{figure}[h]
    \input{pic-link-14.tex}   \quad 
    \input{pic-link-15.tex}    
    \caption{case $\mathcal{E}$: a family of ellipses (left) ; 
    case $\mathcal{H}$: a family of hyperbolas (right).
    \label{fig:familyconics}
    }
  \end{figure}
 
\bigskip
  
\textbf{Step 3.} \emph{Local equation of the projection.}

The projection $\pi_r$ is defined by $\pi_r(x,y) = \frac{x}{r + y}$
(where $(x,y) \in S_r$).
We expand this for $r$ fixed near $r_0$ to get a local 
parameterization of the projection:
\begin{align*}
\pi_r\big( x(t), y(t) \big)
& = \frac{x(t)}{r + y(t)}
= \frac{x_0 + \alpha t + \gamma t^2 + O(t^3)}{r + y_0 + \beta t + \delta t^2 + O(t^3)} \\
& = \frac{x_0}{r+y_0} 
+ \frac{\alpha(r+y_0)-\beta x_0}{(r+y_0)^2} t 
+ c_r t^2  + \cdots
\end{align*}

But $(x_0,y_0)$ is assumed to be a critical point of the projection $\pi_{r_0}$
so by Lemma \ref{lem:critproj}, $x_0 \frac{\partial f}{\partial x}(x_0,y_0) + 
(r_0+y_0)\frac{\partial f}{\partial y}(x_0,y_0) = 0$, i.e.
$\alpha(r_0+y_0)-\beta x_0=0$.
Then for $r=r_0$:
$$\pi_{r_0}\big( x(t), y(t) \big)
= \frac{x_0}{r_0+y_0} 
+ c_{r_0} t^2  + \cdots
$$
where $c_{r_0} = \frac{\gamma(r_0+y_0)-\delta x_0}{(r_0+y_0)^2}$.

We suppose that the projection is non degenerate, that is,
\begin{equation}\label{eq:c0nonzero}
c_{r_0}\neq 0\ \iff\ \gamma(r_0+y_0)-\delta x_0\neq 0,
\end{equation}
see Section~\ref{sec:afewwords} for more details.

This condition implies that $c_r\neq 0$ for $r$ sufficiently close to $r_0$.
The projection is then given by a map $t \mapsto a_r + 
b_rt+c_rt^2$ where $b_r$ is small (because $r\sim r_0$),
so that up to a translation and homothety we model the projection $\pi_r$ by the map 
$t \mapsto t^2$.

\bigskip

\textbf{Step 4.} \emph{Effect of the projection.}

\smallskip

Remember that in Step 2 we found the geometry for the local behaviour of the parameters $t$: a family of ellipses or hyperbolas
defined by \eqref{eq:conics}. Looking at local parametrization of these conics near the point $t=(0,0)\in\Cc$, we infer that the
family of these conics is locally given (up to homothety)
by $(\kappa - s^2,s)$; 
$\kappa$ is for the real parameter of the family of links and depends on $r$; $s$ is a real parameter for 
the parameterization of the conic, see Figure \ref{fig:familyconics}.

The local model of the projection $t \mapsto t^2$, calculated in Step~3,
sends the local model of the parametrization of the link $L_r$,
 to $( (\kappa-s^2)^2 - s^2, 2\kappa s-2s^3)\in\Cc$.
When $\kappa$ goes from negative values to positive ones, it corresponds to 
$\reid_1^{\uparrow}$ (see Figure \ref{fig:projreid1}). 
But the sign of $\kappa$ does not have to be equal to the sign of $r-r_0$. In fact, the sign of $\kappa$
depends on whether changing from $r<r_0$ to $r>r_0$ induces a move of the conic in the family to the right or to the left;
see the discussion at the end of Step 2. We get:
\begin{itemize}
  \item case $\mathcal{E}$: the sign of $\kappa$ is the sign of $r-r_0$; so that the 
move is $\reid_1^{\uparrow}$ when $r$ grows;
  \item case $\mathcal{H}$: the sign of $\kappa$ is opposite to the sign of $r-r_0$; 
so that the move is $\reid_1^{\downarrow}$ when $r$ grows.
\end{itemize}

  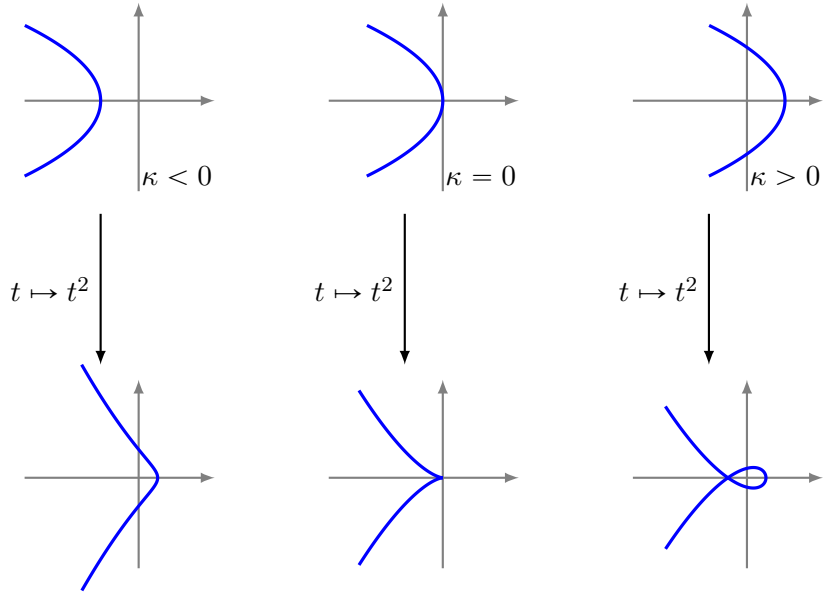
\begin{figure}[h]
    \input{pic-link-19.tex}
    \caption{The Reidemeister 1 move. Top: the parameters $t$; bottom: local behavior 
    of the link diagram.\label{fig:projreid1}}
  \end{figure}

\bigskip
\textbf{Step 5.} \emph{The type $\boxplus$/$\boxminus$ of crossings.}

We want to determine the type of the crossing $\boxplus$/$\boxminus$
for the Reidemeister move $\reid_1$, that is to say we need to determine which 
branch is above the other.
The type of crossing is determined by the last coordinate of the map 
$\Psi_r(x,y) =
\left(\frac{x}{r+y}, \frac{+r \Im(y)}{|r+y|^2}\right)$.
We define two points $t_{+}$, $t_{-}$ of the set of parameters of the link 
$L_r$ that are sent to the double point by the map $\pi_r$. 

For simplicity we will make the computation at the point $(x_0,y_0)=(1,0) \in S_1$, with $r_0=1$.
By Lemma~\ref{lem:critproj}, as we supposed
$(x_0,y_0)$ is a critical point, we have 
$x_0 \frac{\partial f}{\partial x}(x_0,y_0) + (r_0+y_0)\frac{\partial 
f}{\partial y}(x_0,y_0) = 0$,
i.e. $\alpha(r_0+y_0)-\beta x_0=0$.
So for our point $(x_0,y_0) = (1,0)$ we get $\alpha=\beta$.
Moreover $\omega = \ol{x_0}\alpha+\ol{y_0}\beta = \alpha = \beta$.

For our model of projection $t\mapsto t^2$ and our model of conics, 
we have $t_-=-t_+$.
We express these points $t_-=-t_+$ with respect to $\omega$:
$$t_+ = + \ii \epsilon \ol \omega, \qquad 
t_- = - \ii \epsilon \ol \omega,$$
where $\epsilon$ depends on $r-r_0$. The sign of $\epsilon$ depends on the cases:
\begin{itemize}
  \item case $\mathcal{E}$: $\epsilon>0$;
  \item case $\mathcal{H}$: $\epsilon <0$.
\end{itemize}

  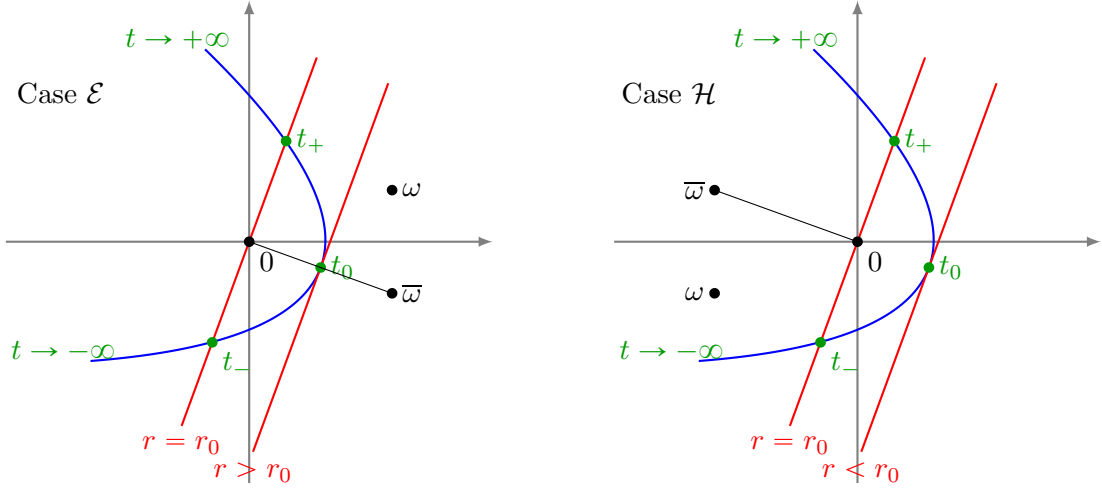
\begin{figure}[h]
    \input{pic-link-20.tex}
    \caption{The parameters $t_+$, $t_-$.}
  \end{figure}

We compute for $y(t)= y_0 + \beta t + \delta t^2 + \cdots$, the sign of the last coordinate of $\Psi_r$, denoted here by $Z$:
\begin{align*}
\sgn(Z(t_+)) &= \sgn \left(\frac{+r \Im(y(t_+))}{|r+y(t_+)|^2}\right) 
= \sgn \left( \Im(y(t_+)) \right)\\
&= \sgn \left( \Im(y_0 + \beta \cdot \ii \epsilon \ol\omega + \delta (\ii \epsilon \ol\omega)^2 + \cdots ) \right).
\end{align*}
Since $y_0=0$ and $\beta\ol\omega = |\beta|^2$ and $\epsilon$ is small, we get
\begin{align*}
\sgn(Z(t_+))& = \sgn(\epsilon).\\
\intertext{In a similar way:}
\sgn(Z(t_-))& = -\sgn(\epsilon).
\end{align*}
So that:
\begin{itemize}
  \item case $\mathcal{E}$ ($\epsilon>0$): $Z(t_+)>Z(t_-)$, the crossing 
is of type $\boxplus$\,;
  \item case $\mathcal{H}$ ($\epsilon<0$): $Z(t_+)<Z(t_-)$, the crossing 
is of type $\boxminus$\,.
\end{itemize}

Pictures for case $\mathcal{E}$ and case $\mathcal{H}$ are in Figures~\ref{fig:Ecase}.
  \begin{figure}[h]
    \input{pic-link-21.tex}
    \caption{The parameters $t_+$, $t_-$ and their projections.}\label{fig:Ecase}
  \end{figure}
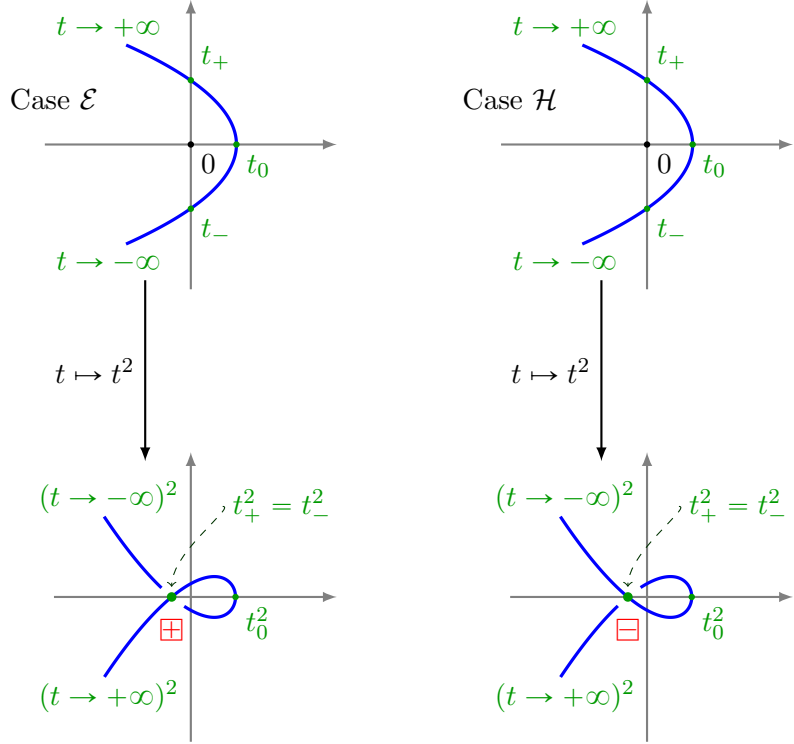

\bigskip
\textbf{Step 6.} \emph{The type $\oplus$/$\ominus$ of loops.}

To compute orientation  $\oplus$/$\ominus$ of the loop
of the move $\reid_1$, we choose a positively oriented tangent vector $v$ 
to the link $L_r$, and compute its image by the projection $\pi_r$.

Let $v_0=(x'(t),y'(t))$ be a tangent vector to $L_r$
(this a vector collinear to the vector
$\left(-\frac{\partial f}{\partial y} (z),
  \frac{\partial f}{\partial x} (z)\right)$).
We define 
  $$v = +\ii \; \overline{\langle v_0 | z \rangle_\Cc} \cdot v_0,$$
  which is a tangent vector to $L_r$ at $z=(x(t),y(t))$, oriented positively, 
with the orientation
  on $L_r$ induced by $L_r$ being the boundary of $\mathcal{C}\cap B_r$.
  
Recall that the diagram of the link is obtain by the projection $\pi_r$ 
corresponding to the 
first (complex) component of the map $\Psi_r$ so that
$\pi_r(x,y) = \frac{x}{r+y}$.
The differential of $\pi_r$ is
$$\dd \pi_r(z)(v_x,v_y) = \frac{1}{r+y}v_x - \frac{x}{(r+y)^2} v_y.$$

To determine the orientation of the loop we first need to compute the projection 
of the tangent vector at one point. We choose the point whose parameter is 
$t_0 = +\epsilon \ol\omega$, where the sign of $\epsilon$ depends on our 
cases:
\begin{itemize}
  \item case $\mathcal{E}$: $\epsilon>0$;
  \item case $\mathcal{H}$: $\epsilon<0$.
\end{itemize}

  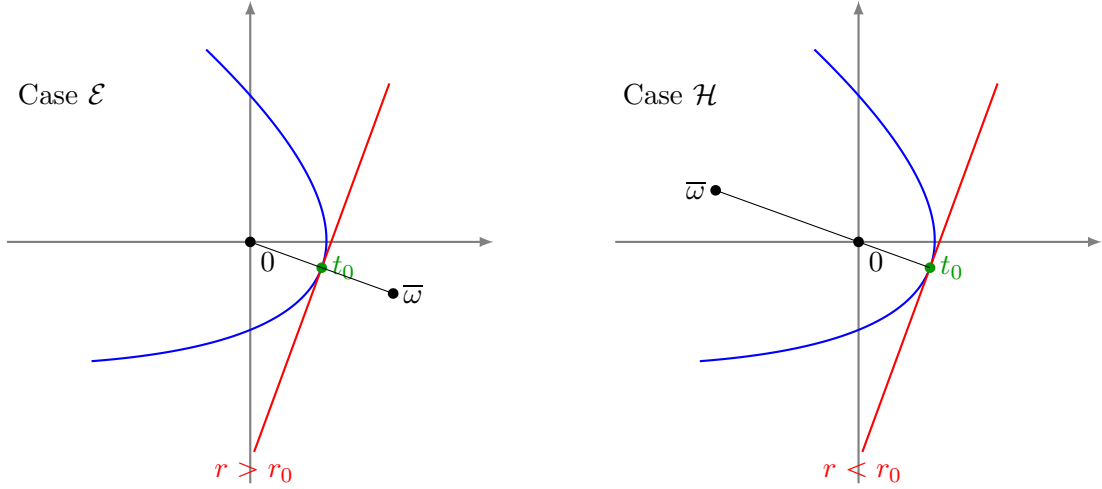
\begin{figure}[h]
    \input{pic-link-22.tex}
    \caption{The point $t_{0}$.}
  \end{figure}

For this $t_0 = \epsilon \ol \omega$, we have
$$\left\{\begin{array}{rcl}
x(t_0) &=& x_0 + \alpha t_0 + \gamma t_0^2 + \cdots = x_0 + \epsilon\alpha\ol\omega + \cdots \\
y(t_0) &=& y_0 + \beta t_0 + \delta t_0^2 + \cdots  = y_0 + \epsilon\beta\ol\omega + \cdots \\
\end{array}\right.$$
where the dots are for higher order terms (with respect to $\epsilon$).
So that
$$r = \sqrt{|x(t_0)|^2+|y(t_0)|^2}
= r_0 + \epsilon \frac{|\omega|^2}{r_0} + \cdots.$$
Moreover:
$$\left\{\begin{array}{rcl}
x'(t_0) &=& \alpha + 2\gamma t_0 + \cdots = \alpha + 2\epsilon \gamma\ol\omega + \cdots\\
y'(t_0) &=& \beta + 2\delta t_0 + \cdots = \beta + 2\epsilon \delta\ol\omega + \cdots\\ 
\end{array}\right.$$
Now $v_0=\begin{pmatrix}x'(t_0)\\y'(t_0)\end{pmatrix}$ and
$$v = +\ii \; \overline{\langle v_0 | z \rangle_\Cc} \cdot v_0
 =  +\ii (\ol\omega + \cdots)
\begin{pmatrix} 
\alpha + 2\epsilon\gamma{\ol\omega}+\cdots\\ 
\beta  + 2\epsilon\delta{\ol\omega}+\cdots\\ 
\end{pmatrix}.$$
Applying the differential, with $v = (v_x,v_y)$ and $z_0 = (x(t_0),y(t_0))$ we get:
\begin{align*}
\dd \pi_r(z_0)(v) 
  &= \frac{1}{(r+y)^2}\left((r+y)v_x - xv_y\right) \\
  &=  \ii \frac{\ol\omega+\cdots}{(r+y)^2}
  \Big(
  (r_0 + \epsilon \frac{|\omega|^2}{r_0} + y_0 + \epsilon\beta\ol\omega + \cdots)(\alpha + 2\epsilon\gamma{\ol\omega}+\cdots)\\\
  & \qquad 
  -(x_0 + \epsilon\alpha\ol\omega + \cdots)(\beta  + 2\epsilon\delta{\ol\omega}+\cdots)
  \Big)\\
  &=  \ii \epsilon \frac{{\ol\omega}^2}{(r_0+y_0)^2}
  \left( \frac{\alpha\omega}{r_0} + 2 \big( (r_0+y_0)\gamma - x_0\delta\big) \right) + \cdots\\
\end{align*}
where we used  $\alpha(r_0+y_0)-\beta x_0=0$.

The direction of the vector depends on the sign of $\epsilon$.
To determine the orientation of the loop in the 
diagram $D_r$, we compute a normal outward vector to the diagram at $\pi_r(z_0)$.
Remember that $t_0 = +\epsilon\ol\omega$, and $z_0 = ( x(t_0),y(t_0) )$.
We denote the projection $\pi_r(\epsilon) = \pi_r(x(\epsilon\ol\omega),y(\epsilon\ol\omega))$ as a function of $\epsilon$.
An inner point of the loop, is $\pi_r(\epsilon(1-\eta))$, for some $0 < \eta \ll 1$.
An outward vector $w':=\overrightarrow{\pi_r(\epsilon-\epsilon\eta)\pi_r(\epsilon)}$ is
given by the complex number $\pi_r(\epsilon) - \pi_r(\epsilon-\epsilon\eta)$.
But
$$\frac{\dd \pi_r(\epsilon)}{\dd \epsilon} (\epsilon) 
= \lim_{\eta'\to 0} \frac{\pi_r(\epsilon) - \pi_r(\epsilon-\eta')}{\eta'}$$
We see that up to higher order terms in $\epsilon$, the vector $w'$ is positively linearly equivalent to
$w = \epsilon \frac{\dd \pi_r(\epsilon)}{\dd \epsilon} (\epsilon)$.

  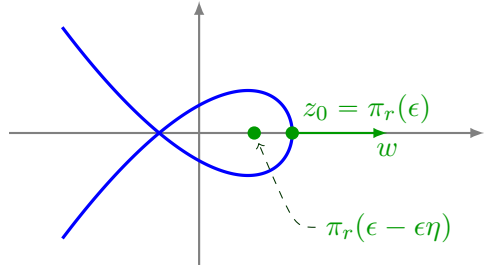
\begin{figure}[h]
    \input{pic-link-28.tex}
    \caption{The outward vector $w$.}
  \end{figure}

Now
$$\pi_r(\epsilon)
= \pi_r(z_0)
= \frac{x_0+\alpha t_0+\gamma t_0^2+ \cdots}{r+y_0 + \beta t_0 + \delta t_0^2 + \cdots}
= \frac{x_0+\epsilon\alpha\ol\omega +\epsilon^2\gamma{\ol\omega}^2+ \cdots}{r+y_0 + \epsilon\beta\ol\omega
+\epsilon^2 \delta {\ol\omega}^2 + \cdots}$$

For the fixed radius $r$ and using again the relation  $\alpha(r_0+y_0)-\beta x_0=0$, we get :
$$w =  \epsilon \frac{\dd \pi_r(\epsilon)}{\dd \epsilon} (\epsilon) 
= \epsilon^2 \frac{{\ol\omega}^2}{(r_0+y_0)^2}
  \left( \frac{\alpha\omega}{r_0} + 2 \big( (r_0+y_0)\gamma - x_0\delta\big) \right) + \cdots$$
So that the lowest order term (with respect to $\epsilon$) of $v$ equals to the lowest order term of $\epsilon \ii w$.

In case $\mathcal{E}$, $\epsilon>0$, $(w,v)$ is a positive basis, and the loop that appears 
in the move $\reid_1^{\uparrow}$ is positively oriented; see Figure~\ref{fig:pic23} (left side). 
In case $\mathcal{H}$, $\epsilon<0$, $(w,v)$ is a negative basis and the loop that 
disappears in the move $\reid_1^{\downarrow}$ is negatively oriented; see Figure~\ref{fig:pic23} (right side).

  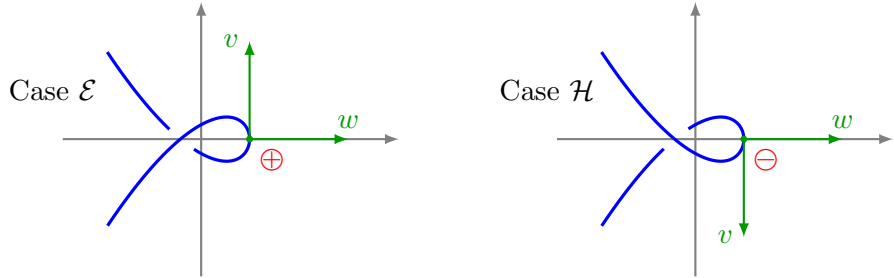
\begin{figure}[h]
    \input{pic-link-23.tex}
    \caption{The orientation $\oplus$/$\ominus$ of the loop.}\label{fig:pic23}
  \end{figure}
  
\bigskip
\textbf{Step 7.} \emph{Conclusion.}

We have separated our studies in two cases:
\begin{itemize}
  \item Case $\mathcal{E}$, $\epsilon>0$. Then when $r$ goes to $r_0-\epsilon$ to $r_0+\epsilon$ 
  the move for the diagram of the link $L_r$ is $\reid_1^{\uparrow\boxplus\oplus}$.
  \item Case $\mathcal{H}$, $\epsilon<0$. Then the move is $\reid_1^{\downarrow\boxminus\ominus}$.
\end{itemize}

\end{proof}

\subsection{A few words on genericity of $\mathcal{C}$}\label{sec:afewwords}

Throughout the paper we were assuming that $\mathcal{C}$ is generic. Here we resume all the genericity conditions that were used. We 
refer to the book of Arnold \cite{Arn} for a thorough treatment of genericity condition.

We begin with the standard one.
\begin{condition}\label{con:milnorgenericity}
The distance function $(x,y)\mapsto |x|^2+|y|^2$ on $\Cc^2$ restricts to a Morse function on $\mathcal{C}$.
\end{condition}
This condition is shown to be generic in \cite[Theorem 6.6]{Mi-Morse}. In fact, by translating the coordinate center by a
generic vector of arbitrary small length (or translating $\mathcal{C}$ in the opposite direction)
we can guarantee Condition~\ref{con:milnorgenericity}. For fixed $r$, if the translation
vector is very small, the translation does not change the isotopy type of $D_r$ (as long as $D_r$ has only double points as singularities and
$\mathcal{C}$ intersects $S_r$ transversally).

Notice that Condition~\ref{con:milnorgenericity} is used in Section~\ref{sec:preliminaries}. It is equivalent to
saying that $\det H\neq 0$ at the intersection points of $\mathcal{C}$ with $\{Jf(x,y)=0\}$.

We have another condition that is easy to satisfy.
\begin{condition}\label{con:almosttrivial}
The curves $\mathcal{C}$, $\{Jf(x,y)=0\}$ and $\{x\frac{\partial f}{\partial x}+((|x|^2+|y|^2)^{1/2}+y)\frac{\partial f}{\partial y}=0\}$ do
not have any common intersection point.
\end{condition}
\begin{lemma}\label{lem:perturbing}
Condition~\ref{con:almosttrivial} can be guaranteed by perturbing $f$ by a linear term.
\end{lemma}
\begin{proof}
By adding a generic term of type $ax+by$ to $f$ we can assume that $\{Jf(x,y)=0\}$ and $\{x\frac{\partial f}{\partial x}+((|x|^2+|y|^2)^{1/2}+y)\frac{\partial f}{\partial y}=0\}$ 
have only a finite number of intersection points. Let $a_1,\ldots,a_n$ be the values of $f$ at these intersection points. If none of these values
is $0$, then Condition~\ref{con:almosttrivial} is satisfied. If one of the $a_i$ is $0$, we replace $f$ by $f+\varepsilon$. Notice that
this change does not affect $Jf$ nor $\{x\frac{\partial f}{\partial x}+((|x|^2+|y|^2)^{1/2}+y)\frac{\partial f}{\partial y}=0\}$, because they
depend only on the derivatives of $f$.
\end{proof}

The next genericity condition is used in Step 2 of Section~\ref{sec:proofofreid1}.
\begin{condition}\label{con:finitenumber}
The map $\mathcal{C}\to \Rr^2\times[0,\infty)$ given by $(x,y)\to (\pi_r(x,y),r)$ has only generic singularities. A
generic singularity means that singularities are isolated and the equation~\eqref{eq:c0nonzero} is satisfied at each of the critical points.
\end{condition}
\begin{lemma}
If $\deg f\ge 2$, the Condition~\ref{con:finitenumber} is generic, that is, adding a generic constant term to $f$
makes $\mathcal{C}$ satisfy Condition~\ref{con:finitenumber}.
\end{lemma} 
\begin{proof}[Sketch of proof]
Recall that the condition \eqref{eq:c0nonzero} means that $\gamma(r_0+y_0)-\delta x_0\neq 0$, where $(x_0,y_0)$ is the critical
point of $\pi_r$, $r_0=(|x_0|^2+|y_0|^2)^{1/2}$ and $\gamma,\delta$ are second order terms of the expansion of a local parametrization
of $\mathcal{C}$ near the point $(x_0,y_0)$. Therefore \eqref{eq:c0nonzero} can be rephrased as a condition involving first and second
derivatives of $f$. We will write this condition as
\[Pf(x_0,y_0)\neq 0,\]
where $P$ is some second order linear differential operator. If $Pf$ is not identically zero, then $Pf$ and $x\frac{\partial f}{\partial x}+((|x|^2+|y|^2)^{1/2}+y)\frac{\partial f}{\partial y}$ have only finitely many common roots. As in the proof of Lemma~\ref{lem:perturbing} we take $a_1,\ldots,a_n$ to be the values of $f$ at these roots and then
replace $\mathcal{C}=f^{-1}(0)$ by $f^{-1}(\xi)$ for $\xi$ sufficiently small and $\xi\notin\{a_1,\ldots,a_n\}$. 
\end{proof}

Condition~\ref{con:finitenumber} on singular points of $\pi_r|_{\mathcal{C}}$ implies that these singular points are non-degenerate. In
particular we have the following corollary.
\begin{corollary}
Singular points of $\pi_r|_{\mathcal{C}}$ are isolated.
\end{corollary}

We also need a variant of Condition~\ref{con:milnorgenericity}.
\begin{condition}\label{con:milnorimproved}
The Hessian of $f$ is non-degenerate at all critical points of $\pi_r|_{\mathcal{C}}$.
\end{condition}
The proof of genericity of this condition (with respect to translating $\mathcal{C}$ by a small vector) is the same as the proof of 
\cite[Theorem 6.6]{Mi-Morse} stated in Condition~\ref{con:milnorgenericity} above. We have \emph{a priori} countably many 
critical points, so genericity is a dense condition, not necessarily
open-dense (we obtain an open-dense set of possible translate vectors that guarantee that the Hessian is non-generate at first $N$ critical 
points of $\pi_r|_{\mathcal{C}}$,
we pass with $N$ to infinity and apply Cantor's lemma).

However, we can have another result that follows from genericity condition. This result uses Theorem~\ref{th:reid1}, but we do not 
use it in the proof of Theorem~\ref{th:reid1}.

\begin{corollary}
There are only finitely many singular points of $\pi_r|_{\mathcal{C}}$.
\end{corollary}
\begin{proof}
The proof of Theorem~\ref{th:reid1} does not need that $\pi_r|_{\mathcal{C}}$ has finitely many critical points, only that they are non-degenerate
in the sense that Condition~\ref{con:finitenumber} and Condition~\ref{con:milnorimproved} are satisfied. Applying Theorem~\ref{th:reid1} implies
that each critical point of $\pi_r|_{\mathcal{C}}$ increases the writhe of the corresponding diagram. As the link at infinity has finite writhe,
the number of critical points is always finite.
\end{proof}



\end{document}

%% file: pic-link-26.tex
\def\myepsilon{0.2}
\newcommand{\pluscrossing}{
\draw[thick, black] (-1,1)--(1,-1);
\draw[thick, black] (-1,-1)--(-\myepsilon,-\myepsilon);
\draw[thick, black] (\myepsilon,\myepsilon)--(1,1);
}
\newcommand{\minuscrossing}{
\draw[thick, black] (1,1)--(-1,-1);
\draw[thick, black] (-1,1)--(-\myepsilon,\myepsilon);
\draw[thick, black] (\myepsilon,-\myepsilon)--(1,-1);
}
\newcommand{\knotbot}{
(-1,-1)
.. controls +(-1,-1) and +(-1,0) .. (0,-3.5)
.. controls +(1,0) and +(1,-1) .. (1,-1)
}

\newcommand{\knotleft}{
(-1,1)
.. controls +(-1,1) and +(1,1) .. (-4,1)
.. controls +(-1,-1) and +(-0.5,1) .. (-4.5,-4)
.. controls +(0.5,-1) and +(-1,1) .. (-1,-7)
}
\newcommand{\knotright}{
(1,1)
.. controls +(1,1) and +(-1,1) .. (4,1)
.. controls +(1,-1) and +(0.5,1) .. (4.5,-4)
.. controls +(-0.5,-1) and +(1,1) .. (1,-7)
}

\newcommand{\fillknotbot}{
(0,0)
--(-1,-1)
.. controls +(-1,-1) and +(-1,0) .. (0,-3.5)
.. controls +(1,0) and +(1,-1) .. (1,-1)
--cycle
}

\newcommand{\fillknottop}{
(0,0)--\knottop--(0,0)
-- (1,-1)
.. controls +(1,-1) and +(1,0) .. (0,-3.5)
.. controls +(-1,0) and +(-1,-1) .. (-1,-1)

}

\newcommand{\fillknot}{
(0,0)--\knottop--(0,0)
}

\begin{tikzpicture}[scale=0.3]
            \pgfdeclareradialshading{glow}{\pgfpoint{0cm}{0cm}}{
 			 color(0mm)=(white);
  			color(3mm)=(white);
  			color(7mm)=(black);
  			color(10mm)=(black)
			}

			\begin{tikzfadingfrompicture}[name=glow fading]
  				\shade [shading=glow] (0,0) circle (1);
			\end{tikzfadingfrompicture}

\begin{scope}[rotate=0]

\draw[thick, blue] \knotbot;
\draw[thick, blue] (-1,1)--(1,-1);
\draw[thick, blue] (-1,-1)--(-\myepsilon,-\myepsilon);
\draw[thick, blue] (\myepsilon,\myepsilon)--(1,1);
\draw[thick, blue] \knotleft;
\draw[thick, blue] \knotright;
\draw[thick, blue] (-1,-7)--(0,-8);
\draw[thick, blue] (-1,-1)--(-\myepsilon,-\myepsilon);
\draw[thick, blue] (\myepsilon,\myepsilon)--(1,1);
\draw[thick, blue] (\myepsilon,-8+\myepsilon)--(1,-7);

\draw[thick, dashed, red] (-1,-9)--(-\myepsilon,-8-\myepsilon);
 \draw[thick, dashed, red] (0,-8)--(1,-9);

 \begin{scope}[yshift=-8cm]
  \draw[thick, dashed, red] \knotbot;
  \end{scope}

\draw[->,>=latex,blue, very thick] (4.8,-2)--+(0,0.3) node[below right]{$\gamma$};

\node at (0,-2) [blue,fill=white,inner sep=8pt,path fading=glow fading]{$A_2$};
\node at (0,-5) [blue,fill=white,inner sep=8pt,path fading=glow fading]{$A_1$};

\node at (0,-10) [red,fill=white,inner sep=8pt,path fading=glow fading]{$A'_1$};
\draw[->,>=latex,red, very thick, dashed] (0,-11.5)--+(-0.3,0) node[below right]{$\gamma'$};
\end{scope}

\end{tikzpicture}

%% file: pic-link-27.tex
\begin{tikzpicture}[scale=1]

\def\myepsilon{0.1}

\draw[thick, black] (-1,0.5)--(0,0);
\draw[->,>=latex,thick, black] (\myepsilon,\myepsilon/2)--(1,0.5);

\draw[thick, red] (-1,-0.5)--(-\myepsilon,-\myepsilon/2);
\draw[->,>=latex,thick, red] (0,0)--(1,-0.5);

\draw[gray] (-2.5,-1)--++(4,0)--++(1,2)--++(-4,0)--cycle;

\begin{scope}[yshift=2cm]
  \draw[thick, black] (-1,0.5)--(1,-0.5);

  \draw[thick, black] (-1,-0.75)-- +(1.14,0.57) ;
  \draw[thick, black] (1,0.25)--+(-0.7,-0.35);

  \draw[thick, red] (-1,-0.75)-- +(1,0.5) node[midway, below] {$\tilde\alpha_i$};
  \draw[thick, red] (0,0)--(1,-0.5) node[midway, below right] {$\tilde\alpha_{i+1}$};

  \draw[thin,gray] (0,0)--+(-1,-0.1);
  \draw[thin,gray] (0,-0.25)--+(-1,-0.1);
 \draw[<->,thin,gray] (-1,-0.1)--(-1,-0.35) node[midway,left, black] {$\lambda_i$};
\end{scope}

\draw[thin, gray] (0,0)--+(0,2);

\fill[blue] (0,0)  circle (1.5pt) node[above, black]{$w_i$};

\node at (1.5,0) {$D_r$};
\node at (1.5,2) {$\tilde L_r$};
\node[red] at (1,-0.75) {$\gamma$};

\fill[blue] (0,2) circle (1.5pt);
\fill[blue] (0,1.75) circle (1.5pt);

\end{tikzpicture}

%% file: pic-link-25.tex
\usetikzlibrary{arrows,decorations.markings}
\begin{tikzpicture}[scale=1]

%


\newcommand{\knot}{
(0,0)
.. controls +(-1,1) and +(-0.5,-1) .. (0,2)
.. controls +(0.5,1) and +(-1,0.7) .. (2,2)
.. controls +(1,-0.7) and +(1,1) .. (1.5,0.5)
.. controls +(-1,-1) and +(1,-1) .. (0,0)
}

\begin{scope}[rotate=-10]
\draw[thick, black] \knot;

\tikzset{myptr/.style={decoration={markings,mark=at position 1 with %
    {\arrow[scale=2.5,>=latex]{<}}},postaction={decorate}}}
    
\draw[myptr] (0,0)--+(+0.1,-0.1) node[below left]{$\gamma$};
 \end{scope}

\end{tikzpicture}

%% file: pic-link-32.tex
\def\myepsilon{0.2}
\newcommand{\pluscrossing}{
\draw (-1,1)--(1,-1);
\draw  (-1,-1)--(-\myepsilon,-\myepsilon);
\draw (\myepsilon,\myepsilon)--(1,1);
}
\newcommand{\minuscrossing}{
\draw (1,1)--(-1,-1);
\draw (-1,1)--(-\myepsilon,\myepsilon);
\draw (\myepsilon,-\myepsilon)--(1,-1);
}
\newcommand{\knotbot}{
(-1,-1)
.. controls +(-1,-1) and +(-1,0) .. (0,-3.5)
.. controls +(1,0) and +(1,-1) .. (1,-1)
}

\newcommand{\knotleft}{
(-1,1)
.. controls +(-1,1) and +(1,1) .. (-4,1)
.. controls +(-1,-1) and +(-0.5,1) .. (-4.5,-4)
.. controls +(0.5,-1) and +(-1,1) .. (-1,-7)
}
\newcommand{\knotright}{
(1,1)
.. controls +(1,1) and +(-1,1) .. (4,1)
.. controls +(1,-1) and +(0.5,1) .. (4.5,-4)
.. controls +(-0.5,-1) and +(1,1) .. (1,-7)
}

\newcommand{\fillknotbot}{
(0,0)
--(-1,-1)
.. controls +(-1,-1) and +(-1,0) .. (0,-3.5)
.. controls +(1,0) and +(1,-1) .. (1,-1)
--cycle
}

\newcommand{\fillknottop}{
(0,0)--\knottop--(0,0)
-- (1,-1)
.. controls +(1,-1) and +(1,0) .. (0,-3.5)
.. controls +(-1,0) and +(-1,-1) .. (-1,-1)
}

\newcommand{\fillknot}{
(0,0)--\knottop--(0,0)
}


\begin{tikzpicture}[scale=0.3]
\begin{scope}[rotate=130, blue, very thick]

\draw \knotbot; 
\pluscrossing;
\draw \knotleft;
\draw \knotright;

    \begin{scope}[yshift=-8cm]
         \pluscrossing;
         \draw \knotbot;
   \end{scope}

\end{scope}

\end{tikzpicture}

%% file: pic-link-11.tex
\begin{tikzpicture}[scale=0.8]

\draw[black, thick] circle(1cm);
\node at (0,-1.5) {$\oplus$};
\draw[->,>=latex,black, very thick] (0,1)--+(-0.3,0);

\begin{scope}[xshift=4cm]
\draw[black, thick] circle(1cm);
\node at (0,-1.5) {$\ominus$};
\draw[->,>=latex,black, very thick] (0,1)--+(0.3,0);

\end{scope}

\end{tikzpicture}

%% file: pic-link-12.tex
\begin{tikzpicture}[scale=0.5]

\draw[dashed, black, thick, rounded corners] (-0.5,-1.5) rectangle +(1,3);
\draw[->,>=latex,thick, black]  (-2,2.5).. controls +(1,-1) and +(-1,-1) .. (2,2.5);
\draw[<-,>=latex,thick, black]  (-2,-2.5).. controls +(1,1) and +(-1,1) .. (2,-2.5);

\draw[->,>=latex,black, very thick] (0.5,0)--+(0,0.3);
\draw[->,>=latex,black, very thick] (-0.5,0)--+(0,-0.3);
\node at (0,0) {$\oplus$};

\begin{scope}[xshift=7cm]
\draw[dashed, black, thick, rounded corners] (-0.5,-1.5) rectangle +(1,3);
\draw[<-,>=latex,thick, black]  (-2,-2.5).. controls +(1,0.5) and +(0,-0.5) .. (-0.7,-1) -- (-0.7,1)
.. controls +(0,0.5) and +(1,-0.5) ..(-2,2.5);
\draw[->,>=latex,thick, black]  (2,-2.5).. controls +(-1,0.5) and +(0,-0.5) .. (0.7,-1) -- (0.7,1)
.. controls +(0,0.5) and +(-1,-0.5) ..(2,2.5);
\end{scope}
\draw[->,>=latex,blue, very thick] (2,0)--+(3,0) node[midway, above] {$\mathbf{I}^\oplus$};

\begin{scope}[xshift=16cm]
\draw[dashed, black, thick, rounded corners] (-0.5,-1.5) rectangle +(1,3);
\draw[<-,>=latex,thick, black]  (-2,2.5).. controls +(1,-1) and +(-1,-1) .. (2,2.5);
\draw[->,>=latex,thick, black]  (-2,-2.5).. controls +(1,1) and +(-1,1) .. (2,-2.5);

\draw[<-,>=latex,black, very thick] (0.5,0)--+(0,0.3);
\draw[<-,>=latex,black, very thick] (-0.5,0)--+(0,-0.3);
\node at (0,0) {$\ominus$};

\begin{scope}[xshift=7cm]
\draw[dashed, black, thick, rounded corners] (-0.5,-1.5) rectangle +(1,3);
\draw[->,>=latex,thick, black]  (-2,-2.5).. controls +(1,0.5) and +(0,-0.5) .. (-0.7,-1) -- (-0.7,1)
.. controls +(0,0.5) and +(1,-0.5) ..(-2,2.5);
\draw[<-,>=latex,thick, black]  (2,-2.5).. controls +(-1,0.5) and +(0,-0.5) .. (0.7,-1) -- (0.7,1)
.. controls +(0,0.5) and +(-1,-0.5) ..(2,2.5);
\end{scope}
\draw[->,>=latex,blue, very thick] (2,0)--+(3,0) node[midway, above] {$\mathbf{I}^\ominus$};

 \end{scope}    
\end{tikzpicture}

%% file: pic-link-08.tex
\def\myepsilon{0.2}
\newcommand{\pluscrossing}{
\draw[thick, black] (-1,1)--(1,-1);
\draw[thick, black] (-1,-1)--(-\myepsilon,-\myepsilon);
\draw[thick, black] (\myepsilon,\myepsilon)--(1,1);
}
\newcommand{\minuscrossing}{
\draw[thick, black] (1,1)--(-1,-1);
\draw[thick, black] (-1,1)--(-\myepsilon,\myepsilon);
\draw[thick, black] (\myepsilon,-\myepsilon)--(1,-1);
}
\newcommand{\knotbot}{
(-1,-1)
.. controls +(-1,-1) and +(-1,0) .. (0,-3.5)
.. controls +(1,0) and +(1,-1) .. (1,-1)
}
\newcommand{\reidtopleft}{
(-1,1) .. controls +(-1,1) and +(1,0) .. (-6,2)
}
\newcommand{\reidtopright}{
(1,1) .. controls +(1,1) and +(-1,0) .. (6,2)
}

\newcommand{\knotsmooth}{
(-6,2) .. controls +(3,0) and +(-2,0) .. 
(0,0) .. controls +(2,0) and +(-3,0) .. (6,2)
}

\begin{tikzpicture}[scale=0.3]

\begin{scope}[rotate=90]

\begin{scope}[yshift=6cm]
\draw[thick, black] \knotsmooth;
\end{scope}

\draw[thick, black] \knotbot;
\draw[thick, black] \reidtopleft;
\draw[thick, black] \reidtopright;
\pluscrossing;

\draw[->,>=latex,blue, very thick] (1,5)--+(0,-3) node[midway, above] {$\mathbf{\Omega}_1^{\uparrow}$};
\draw[<-,>=latex,blue, very thick] (-1,5)--+(0,-3) node[midway, below] {$\mathbf{\Omega}_1^{\downarrow}$};
 \end{scope}

\end{tikzpicture}

%% file: pic-link-09.tex
\def\myepsilon{0.2}
\newcommand{\pluscrossing}{
\draw[thick, black] (-1,1)--(1,-1);
\draw[thick, black] (-1,-1)--(-\myepsilon,-\myepsilon);
\draw[thick, black] (\myepsilon,\myepsilon)--(1,1);
}
\newcommand{\minuscrossing}{
\draw[thick, black] (1,1)--(-1,-1);
\draw[thick, black] (-1,1)--(-\myepsilon,\myepsilon);
\draw[thick, black] (\myepsilon,-\myepsilon)--(1,-1);
}
\newcommand{\knotbot}{
(-1,-1)
.. controls +(-1,-1) and +(-1,0) .. (0,-3.5)
.. controls +(1,0) and +(1,-1) .. (1,-1)
}
\newcommand{\reidtopleft}{
(-1,1) .. controls +(-1,1) and +(1,0) .. (-6,2)
}
\newcommand{\reidtopright}{
(1,1) .. controls +(1,1) and +(-1,0) .. (6,2)
}

\newcommand{\knotsmooth}{
(-6,2) .. controls +(3,0) and +(-2,0) .. 
(0,0) .. controls +(2,0) and +(-3,0) .. (6,2)
}

\begin{tikzpicture}[scale=0.3]

\begin{scope}[rotate=90]

\begin{scope}[yshift=6cm]
\draw[<-,>=latex,black, very thick] (5,1.97)--+(0.3,0); 
\draw[thick, black] \knotsmooth;
\end{scope}

\draw[thick, black] \knotbot;
\draw[thick, black] \reidtopleft;
\draw[thick, black] \reidtopright;
\pluscrossing;
\draw[<-,>=latex,black, very thick] (5,1.98)--+(0.3,0); 

\draw[->,>=latex,blue, very thick] (1,5)--+(0,-3) node[midway, above] {$\mathbf{\Omega}_1^{\uparrow\boxplus\oplus}$};
\draw[<-,>=latex,blue, very thick] (-1,5)--+(0,-3) node[midway, below] {$\mathbf{\Omega}_1^{\downarrow\boxplus\oplus}$};

\node at (1,0) {$\boxplus$};
\node at (0,-2) {$\oplus$};
 \end{scope}

\begin{scope}[rotate=90, yshift=-18cm]

\begin{scope}[yshift=6cm]
\draw[<-,>=latex,black, very thick] (5,1.97)--+(0.3,0); 
\draw[thick, black] \knotsmooth;
\end{scope}

\draw[thick, black] \knotbot;
\draw[thick, black] \reidtopleft;
\draw[thick, black] \reidtopright;
\minuscrossing;
\draw[<-,>=latex,black, very thick] (5,1.98)--+(0.3,0); 

\draw[->,>=latex,blue, very thick] (1,5)--+(0,-3) node[midway, above] {$\mathbf{\Omega}_1^{\uparrow\boxminus\oplus}$};
\draw[<-,>=latex,blue, very thick] (-1,5)--+(0,-3) node[midway, below] {$\mathbf{\Omega}_1^{\downarrow\boxminus\oplus}$};

\node at (1,0) {$\boxminus$};
\node at (0,-2) {$\oplus$};
 \end{scope}

\begin{scope}[rotate=90, xshift=-15cm]

\begin{scope}[yshift=6cm]
\draw[<-,>=latex,black, very thick] (-5,1.97)--+(-0.3,0); 
\draw[thick, black] \knotsmooth;
\end{scope}

\draw[thick, black] \knotbot;
\draw[thick, black] \reidtopleft;
\draw[thick, black] \reidtopright;
\pluscrossing;
\draw[<-,>=latex,black, very thick] (-5,1.98)--+(-0.3,0); 

\draw[->,>=latex,blue, very thick] (1,5)--+(0,-3) node[midway, above] {$\mathbf{\Omega}_1^{\uparrow\boxplus\ominus}$};
\draw[<-,>=latex,blue, very thick] (-1,5)--+(0,-3) node[midway, below] {$\mathbf{\Omega}_1^{\downarrow\boxplus\ominus}$};

\node at (1,0) {$\boxplus$};
\node at (0,-2) {$\ominus$};
 \end{scope}

\begin{scope}[rotate=90, yshift=-18cm, xshift=-15cm]

\begin{scope}[yshift=6cm]
\draw[<-,>=latex,black, very thick] (-5,1.97)--+(-0.3,0); 
\draw[thick, black] \knotsmooth;
\end{scope}

\draw[thick, black] \knotbot;
\draw[thick, black] \reidtopleft;
\draw[thick, black] \reidtopright;
\minuscrossing;
\draw[<-,>=latex,black, very thick] (-5,1.98)--+(-0.3,0); 

\draw[->,>=latex,blue, very thick] (1,5)--+(0,-3) node[midway, above] {$\mathbf{\Omega}_1^{\uparrow\boxminus\ominus}$};
\draw[<-,>=latex,blue, very thick] (-1,5)--+(0,-3) node[midway, below] {$\mathbf{\Omega}_1^{\downarrow\boxminus\ominus}$};

\node at (1,0) {$\boxminus$};
\node at (0,-2) {$\ominus$};

 \end{scope}
       
\end{tikzpicture}

%% file: pic-link-24.tex
\begin{tikzpicture}[scale=1.1]

\begin{scope}[xshift=0cm]

    \node at (-1.5,0) {$r<r_0$};
  \def\myc{0.8}
   \draw [very thick, color=blue,samples=100,smooth, domain=-1:1]  plot ({-\x*\x}, {\x} );

  \coordinate (P1) at (-0.25,0.5);
  \draw[thick,red] (P1)--+(0.7,-0.7)--+(-0.7,0.7) ;
  \fill[green!60!black] (P1) circle (2pt) node[above right] {$P_1(r)$};

  \coordinate (P2) at (-0.25,-0.5);
  \draw[thick,red] (P2)--+(-0.7,-0.7)--+(0.7,0.7) ;
  \fill[green!60!black] (P2) circle (2pt) node[below right] {$P_2(r)$};

\draw[->,>=latex,thick] (-0.5,-1.5)--(-0.5,-3) node[midway,right]{$\pi_r$};

\begin{scope}[yshift=-4cm]
   \draw [very thick, color=blue,samples=100,smooth, domain=-1.25:1.25]  plot ({0.5 - \x*\x}, {\myc*\x-\x*\x*\x} );

  \coordinate (Q1) at (0.22,0.28);
  \draw[thick,red] (Q1)--+(1,0)--+(-1,0) ;
  \fill[green!60!black] (Q1) circle (2pt) node[above left] {$Q_1(r)$};

  \coordinate (Q2) at (0.5,0);
  \draw[thick,red] (Q2)--+(0,1)--+(0,-1) ;
  \fill[green!60!black] (Q2) circle (2pt) node[below right] {$Q_2(r)$};

\end{scope}

\end{scope}

\begin{scope}[xshift=5cm]

    \node at (-1.5,0) {$r=r_0$};
  \def\myc{0}
   \draw [very thick, color=blue,samples=100,smooth, domain=-1:1]  plot ({\myc - \x*\x}, {\x} );
  \draw[thick,red] (0,-1)--(0,1) ;
  \fill[green!60!black] (0,0) circle (2pt) node[right] {$P_0$};

\draw[->,>=latex,thick] (-0.5,-1.5)--(-0.5,-3) node[midway,right]{$\pi_{r_0}$};

\begin{scope}[yshift=-4cm]
   \draw [very thick, color=blue,samples=100,smooth, domain=-1.05:1.05]  plot ({\myc*\myc - \x*\x}, {\myc*\x-\x*\x*\x} );

  \fill[green!60!black] (0,0) circle (2pt) node[right] {$Q_0$};
\end{scope}

\end{scope}

\end{tikzpicture}

%% file: pic-link-10.tex
\def\myepsilon{0.2}
\newcommand{\pluscrossing}{
\draw[<-,>=latex,black, thick] (1,1)--(-1,-1);
\draw[<-,>=latex,black, thick] (-1,1)--(-\myepsilon,\myepsilon);
\draw[black, thick] (\myepsilon,-\myepsilon)--(1,-1);
}
\newcommand{\minuscrossing}{
\draw[<-,>=latex,black, thick] (-1,1)--(1,-1);
\draw[thick, black] (-1,-1)--(-\myepsilon,-\myepsilon);
\draw[->,>=latex,black, thick] (\myepsilon,\myepsilon)--(1,1);
}

\begin{tikzpicture}[scale=0.8]

\pluscrossing;

\node at (0,-1) {$\boxplus$};
\begin{scope}[xshift=4cm]
\minuscrossing;
\node at (0,-1) {$\boxminus$};
\end{scope}

\end{tikzpicture}

%% file: pic-link-29.tex

\def\myepsilon{0.2}
\newcommand{\pluscrossing}{
\draw[thick, black] (-1,1)--(1,-1);
\draw[thick, black] (-1,-1)--(-\myepsilon,-\myepsilon);
\draw[thick, black] (\myepsilon,\myepsilon)--(1,1);
}
\newcommand{\minuscrossing}{
\draw[thick, black] (1,1)--(-1,-1);
\draw[thick, black] (-1,1)--(-\myepsilon,\myepsilon);
\draw[thick, black] (\myepsilon,-\myepsilon)--(1,-1);
}

\newcommand{\knotbot}{
(-1,1)
.. controls +(-2,2) and +(0,2) .. (-4,-2)
.. controls +(0,-2) and +(-1,0) .. (0,-6)
.. controls +(1,0) and +(0,-2) .. (4,-2)
.. controls +(0,2) and +(2,2) .. (1,1)
}

\newcommand{\knotbotbis}{
(-1,-1)
.. controls +(-2,-2) and +(-0.5,1) .. (-1.5,-5)
.. controls +(0.5,-1) and +(-0.5,0) .. (0,-6)
.. controls +(0.5,0) and +(-0.5,-1) .. (1.5,-5)
.. controls +(0.5,1) and +(2,-2) .. (1,-1)
}

\begin{tikzpicture}[scale=0.5]

\begin{scope}[xshift=0]
\draw[thick] \knotbotbis;
\minuscrossing;

\draw[<-,>=latex, very thick] (2.2,-3.4)--+(0,0.3) node[right]{$\ominus$};
\node at (0,-1) {$\boxminus$};
\end{scope}

\begin{scope}[xshift=10cm]
\draw[thick] \knotbot;
\minuscrossing;

\draw[<-,>=latex, very thick] (4,-2)--+(0,0.3) node[right]{$\ominus$};
\node at (0,-1) {$\boxminus$};
\end{scope}

\end{tikzpicture}

%% file: pic-link-30.tex

\def\myepsilon{0.2}
\newcommand{\pluscrossing}{
\draw[thick, black] (-1,1)--(1,-1);
\draw[thick, black] (-1,-1)--(-\myepsilon,-\myepsilon);
\draw[thick, black] (\myepsilon,\myepsilon)--(1,1);
}
\newcommand{\minuscrossing}{
\draw[thick, black] (1,1)--(-1,-1);
\draw[thick, black] (-1,1)--(-\myepsilon,\myepsilon);
\draw[thick, black] (\myepsilon,-\myepsilon)--(1,-1);
}

\newcommand{\knotmiddle}{
\draw[thick]
(-1,-1)
.. controls +(-0.5,-0.5) and +(0,0.5) .. (-1.5,-2.5)
.. controls +(0,-0.5) and +(-0.5,0.5) .. (-1,-4);
\draw[thick]
(1,-1)
.. controls +(0.5,-0.5) and +(0,0.5) .. (1.5,-2.5)
.. controls +(0,-0.5) and +(0.5,0.5) .. (1,-4);
\draw[dashed,blue] (0,-2.5) ellipse (1 and 2);
\node[blue] at (0,-2.5) {$\ominus$};
}

\newcommand{\knotabove}{
\draw[thick] (-1,1) .. controls +(-0.5,0.5) and +(0,-0.5) .. (-1.5,3);
\draw[thick] (1,1) .. controls +(0.5,0.5) and +(0,-0.5) .. (1.5,3);
\draw[dashed,blue] (-1.2,3.1)
.. controls +(0,-0.5) and +(-0.5,0) .. (0,0.5)
.. controls +(0.5,0) and +(0,-0.5) .. (1.2,3.1);
\node[blue] at (0,2) {$\oplus$};
}

\newcommand{\knotbelow}{
\draw[thick] (-1,-6) .. controls +(-0.5,-0.5) and +(0,0.5) .. (-1.5,-8);
\draw[thick] (1,-6) .. controls +(0.5,-0.5) and +(0,0.5) .. (1.5,-8);
\draw[dashed,blue] (-1.2,-8.1)
.. controls +(0,0.5) and +(-0.5,0) .. (0,-5.5)
.. controls +(0.5,0) and +(0,0.5) .. (1.2,-8.1);
\node[blue] at (0,-7) {$\oplus$};
}

\begin{tikzpicture}[scale=0.5]

\begin{scope}[xshift=0]
  \minuscrossing;
  \begin{scope}[yshift=-5cm]
     \pluscrossing;
  \end{scope}
  \knotmiddle;
  \knotabove;
  \knotbelow;

\draw[<-,>=latex, very thick] (1.5,-2.5)--+(0,0.3);
\draw[->,>=latex, very thick] (-1.5,-2.5)--+(0,0.3);
 \node[red] at (1,0) {$\boxminus$};
 \node[red] at (1,-5) {$\boxplus$};
\end{scope}

\begin{scope}[xshift=10cm]
  \minuscrossing;
  \begin{scope}[yshift=-5cm]
     \minuscrossing;
  \end{scope}
  \knotmiddle;
  \knotabove;
  \knotbelow;

\draw[<-,>=latex, very thick] (1.5,-2.5)--+(0,0.3);
\draw[->,>=latex, very thick] (-1.5,-2.5)--+(0,0.3);
 \node[red] at (1,0) {$\boxminus$};
 \node[red] at (1,-5) {$\boxminus$};
\end{scope}

\end{tikzpicture}

%% file: pic-link-31.tex

\def\myepsilon{0.2}
\newcommand{\pluscrossing}{
\draw[thick, black] (-1,1)--(1,-1);
\draw[thick, black] (-1,-1)--(-\myepsilon,-\myepsilon);
\draw[thick, black] (\myepsilon,\myepsilon)--(1,1);
}
\newcommand{\minuscrossing}{
\draw[thick, black] (1,1)--(-1,-1);
\draw[thick, black] (-1,1)--(-\myepsilon,\myepsilon);
\draw[thick, black] (\myepsilon,-\myepsilon)--(1,-1);
}
\newcommand{\singularcrossing}{
\draw[thick, black] (1,1)--(-1,-1);
\draw[thick, black] (-1,1)--(1,-1);
\fill (0,0) circle (4pt);
}

\newcommand{\knotmiddle}{
\draw[thick]
(-1,-1)
.. controls +(-0.5,-0.5) and +(0,0.5) .. (-1.5,-2.5)
.. controls +(0,-0.5) and +(-0.5,0.5) .. (-1,-4);
\draw[thick]
(1,-1)
.. controls +(0.5,-0.5) and +(0,0.5) .. (1.5,-2.5)
.. controls +(0,-0.5) and +(0.5,0.5) .. (1,-4);
\draw[dashed,blue] (0,-2.5) ellipse (1 and 2);
\node[blue] at (0,-2.5) {$\ominus$};
}

\newcommand{\knotabove}{
\draw[thick] (-1,1) .. controls +(-0.5,0.5) and +(0,-0.5) .. (-1.5,3);
\draw[thick] (1,1) .. controls +(0.5,0.5) and +(0,-0.5) .. (1.5,3);
}

\newcommand{\knotbelow}{
\draw[thick] (-1,-6) .. controls +(-0.5,-0.5) and +(0,0.5) .. (-1.5,-8);
\draw[thick] (1,-6) .. controls +(0.5,-0.5) and +(0,0.5) .. (1.5,-8);
}

\begin{tikzpicture}[scale=0.4]

\begin{scope}[xshift=0cm]
  \minuscrossing;
  \begin{scope}[yshift=-5cm]
     \singularcrossing;
  \end{scope}
  \knotmiddle;
  \knotabove;
  \knotbelow;

\draw[dashed,blue] (-1.2,3.1)
.. controls +(0,-0.5) and +(-0.5,0) .. (0,0.5)
.. controls +(0.5,0) and +(0,-0.5) .. (1.2,3.1);

\draw[dashed,blue] (-1.2,-8.1)
.. controls +(0,0.5) and +(-0.5,0) .. (0,-5.5)
.. controls +(0.5,0) and +(0,0.5) .. (1.2,-8.1);

\draw[<-,>=latex, very thick] (1.5,-2.5)--+(0,0.3);
\draw[->,>=latex, very thick] (-1.5,-2.5)--+(0,0.3);
 \node[red] at (1,0) {$\boxminus$};
\end{scope}

\begin{scope}[xshift=7cm]
  \minuscrossing;
  \begin{scope}[yshift=-5cm]
     \singularcrossing;
  \end{scope}
  \knotabove;
  \knotbelow;

\draw[thick] (1,-1) .. controls +(0.5,-0.5) and +(0,0.5) .. (1.5,-2.5)
.. controls +(0,-0.5) and +(0.5,0.5) .. (1,-4);

    \begin{scope}[xshift=-1.5cm,yshift=-1.5cm,scale=0.5]
     \minuscrossing;
  \end{scope}

\draw[thick] (-2,-1) .. controls +(-0.25,0.25) and +(0,0.5) .. (-2.5,-1.5)
.. controls +(0,-0.5) and +(-0.25,-0.25) .. (-2,-2);

\draw[thick] (-1,-2) .. controls +(0.5,-0.5) and +(-0.50,0.5) .. (-1,-4);

\draw[<-,>=latex, very thick] (1.5,-2.5)--+(0,0.3);
\draw[->,>=latex, very thick] (-1.13,-3.6)--+(0,0.3);
\end{scope}

\begin{scope}[xshift=14cm]
  \minuscrossing;
  \begin{scope}[yshift=-5cm]
     \singularcrossing;
  \end{scope}
  \knotabove;

\draw[thick] (-1,-1)
 .. controls +(-1,-1) and +(-0.5,0.5) .. (1.1,-3.9) 
.. controls +(0.5,-0.5) and +(0.5,0.5) .. (1.1,-6.1)
.. controls +(-0.5,-0.5) and +(0.5,-0.5) .. (-1.1,-6.1)
.. controls +(-0.5,0.5) and +(-0.5,0.5) .. (-1,-4);

\draw[thick]
(1,-1)
.. controls +(0.5,-0.5) and +(0,0.5) .. (1.5,-2.5)
.. controls +(0,-0.5) and +(0.5,0.5) .. (1.2,-3.8);

\draw[thick] (-1.2,-6.2) .. controls +(-0.5,-0.5) and +(0,0.5) .. (-1.5,-8);
\draw[thick] (1.2,-6.2) .. controls +(0.5,-0.5) and +(0,0.5) .. (1.5,-8);

\draw[<-,>=latex, very thick] (1.5,-2.5)--+(0,0.3);
\draw[->,>=latex, very thick] (-1.19,-1.6)--+(0,0.3);
\end{scope}

\begin{scope}[xshift=21cm]

  \begin{scope}[yshift=-5cm]
     \singularcrossing;
  \end{scope}
  \knotabove;

\draw[thick] (1,1)
 .. controls +(-1,-1) and +(-0.5,0.5) .. (1.5,-2) 
.. controls +(0.5,-0.5) and +(0.5,0.5) .. (1.1,-6.1)
.. controls +(-0.5,-0.5) and +(0.5,-0.5) .. (-1.1,-6.1)
.. controls +(-0.5,0.5) and +(-0.5,0.5) .. (-1,-4);

\draw[thick]
(-1,1)
.. controls +(0.5,-0.5) and +(-0.5,1) .. (0.25,-1.5)
.. controls +(0.5,-1) and +(0.5,0.5) .. (1,-4);

\draw[thick] (-1.2,-6.2) .. controls +(-0.5,-0.5) and +(0,0.5) .. (-1.5,-8);
\draw[thick] (1.2,-6.2) .. controls +(0.5,-0.5) and +(0,0.5) .. (1.5,-8);

\draw[<-,>=latex, very thick] (1.7,-4)--+(0,0.3);
\draw[->,>=latex, very thick] (0.26,-1.6)--+(-0.15,0.3);

\draw[dashed,blue] (0,-5.9) ellipse (0.6 and 0.35);
\node[blue] at (0,-5.9) {$\oplus$};

\draw[dashed,blue] (1.7,3.1)
.. controls +(0,-2) and +(0,0.5) .. (1,-0.5)
.. controls +(0,-0.5) and +(0,0.5) .. (2,-2.5)
.. controls +(0,-0.5) and +(0,-0.5) .. (1.7,-8);

\draw[dashed,blue] (-1.7,3.1)
.. controls +(0,-3) and +(0,0.5) .. (0.5,-3.5)
.. controls +(0,-0.5) and +(0,0.5) .. (-2,-3.5)
.. controls +(0,-0.5) and +(0,-0.5) .. (-1.7,-8);
\end{scope}

\draw[->,>=latex, very thick,red] (2,1.5)--(5,1.5) node[midway, above]{$\mathbf{\Omega}_1^{\uparrow\boxplus\oplus}$}; 
\draw[->,>=latex, very thick,red] (9,1.5)--(12,1.5) node[midway, above]{$\mathbf{\Omega}_3$};      
 \draw[->,>=latex, very thick,red] (16,1.5)--(19,1.5) node[midway, above]{$\mathbf{\Omega}_2$};     
\end{tikzpicture}

%% file: stoimenow.tex
\begin{tikzpicture}[scale=0.7]
\begin{scope}[every node/.style={sloped,allow upside down}]
\draw (0.5,-0.5) -- ++ (0,-1.1) -- ++(0.3,-0.3) --  ++(0,-1.4) ++ (-0.3,0.0) -- ++(0,1.4) -- ++ (0.1,0.1) ++ (0.1,0.1) -- ++(0.1,0.1) -- node{\midarrow} ++ (0,1.1);
\draw (-0.8,-0.5) -- ++ (0,-1.1) -- ++(0.3,-0.3) -- ++(0,-1.4) ++ (-0.3,0.0) -- ++(0,1.4) -- ++ (0.1,0.1) ++ (0.1,0.1) -- ++(0.1,0.1) -- ++ (0,1.1);
\draw (0.5,-3.9) -- ++ (0,-1.1) -- ++(0.3,-0.3) -- ++(0,-0.5) ++ (-0.3,0.0) -- ++(0,0.5) -- ++ (0.1,0.1) ++ (0.1,0.1) -- ++(0.1,0.1) -- ++ (0,1.1);
\draw (-0.8,-3.9) -- ++ (0,-1.1) -- ++(0.3,-0.3) -- ++(0,-0.5) ++ (-0.3,0.0) -- ++(0,0.5) -- ++ (0.1,0.1) ++ (0.1,0.1) -- ++(0.1,0.1) -- ++ (0,1.1);
\draw (0.8,4) -- ++ (0,-1.7) -- ++(0.3,-0.3) -- ++(0,-1.4) ++ (-0.3,0.0) -- ++(0,1.4) -- ++ (0.1,0.1) ++ (0.1,0.1) -- ++(0.1,0.1) -- ++ (0,1.7);
\draw (-1.1,4) -- ++ (0,-1.7) -- ++(0.3,-0.3) -- ++(0,-1.4) ++ (-0.3,0.0) -- ++(0,1.4) -- ++ (0.1,0.1) ++ (0.1,0.1) -- ++(0.1,0.1) -- ++ (0,1.7);
\draw (-0.15,6) -- ++ (0,-0.3) -- ++(0.3,-0.3) -- ++(0,-1.4) ++ (-0.3,0.0) -- ++(0,1.4) -- ++ (0.1,0.1) ++ (0.1,0.1) -- ++(0.1,0.1) -- ++ (0,0.3);
%
\draw (1.5,0.2) -- ++ (1.05,0.0) -- ++ (0.4,-0.4) -- ++ (1,0) ++ (0,0.4) -- ++(-0.5,0);
\draw[->] (2.55,0.2) -- ++ (-0.4,0.0);
\draw[color=red, densely dashed,<-] (3.45,0.2) -- ++ (-0.5,0) -- ++ (-0.15,-0.15)  ++ (-0.1,-0.1) -- ++(-0.15,-0.15) -- ++(-0.5,0);
\draw (-3.7,0.4) -- ++ (0.9,0.0) -- ++ (0.3,-0.3) -- ++ (1,0) ++ (0,0.3) -- ++(-1,0) 
-- ++(-0.1,-0.1) ++ (-0.1,-0.1) -- ++ (-0.1,-0.1) -- ++(-0.9,0.0);
\draw (-3.7,-0.25) -- ++ (0.9,0.0) -- ++ (0.3,-0.3) -- ++ (1,0) ++ (0,0.3) -- ++(-1,0) 
-- ++(-0.1,-0.1) ++ (-0.1,-0.1) -- ++ (-0.1,-0.1) -- ++(-0.9,0.0);
\draw (3.9,0.25) -- ++ (0.9,0.0) -- ++ (0.3,-0.3) -- ++ (0.5,0) ++ (0,0.3) -- ++(-0.5,0) 
-- ++(-0.1,-0.1) ++ (-0.1,-0.1) -- ++ (-0.1,-0.1) -- ++(-0.9,0.0);
\draw (0.15,0.95) -- ++(0,0.15) -- ++ (0.2,0.2) -- ++(0,0.2) ++ (-0.2,0) -- ++(0,-0.2) -- ++(0.08,-0.08)  ++(0.04,-0.04) -- ++(0.08,-0.08) -- ++ (0,-0.15);
\draw (-0.35,0.95) -- ++(0,0.15) -- ++ (0.2,0.2) -- ++(0,0.2) ++ (-0.2,0) -- ++(0,-0.2) -- ++(0.08,-0.08)  ++(0.04,-0.04) -- ++(0.08,-0.08) -- ++ (0,-0.15);
\draw (-0.10,1.8) -- ++(0,0.15) -- ++ (0.2,0.2) -- ++(0,0.2) ++ (-0.2,0) -- ++(0,-0.2) -- ++(0.08,-0.08)  ++(0.04,-0.04) -- ++(0.08,-0.08) -- ++ (0,-0.15);
\draw (0.15,2.5) -- ++(0,0.15) -- ++ (0.2,0.2) -- ++(0,0.2) ++ (-0.2,0) -- ++(0,-0.2) -- ++(0.08,-0.08)  ++(0.04,-0.04) -- ++(0.08,-0.08) -- ++ (0,-0.15);
\draw (-0.35,2.5) -- ++(0,0.15) -- ++ (0.2,0.2) -- ++(0,0.2) ++ (-0.2,0) -- ++(0,-0.2) -- ++(0.08,-0.08)  ++(0.04,-0.04) -- ++(0.08,-0.08) -- ++ (0,-0.15);
\draw (-0.35,0.95) -- ++(0,0.15) -- ++ (0.2,0.2) -- ++(0,0.2) ++ (-0.2,0) -- ++(0,-0.2) -- ++(0.08,-0.08)  ++(0.04,-0.04) -- ++(0.08,-0.08) -- ++ (0,-0.15);
\draw [very thick,green!30!black, dotted] plot [smooth] coordinates {(2.55,-0.2) (2.50,-0.5) (2.00,-3.5) (1.0,-4.5) (0.5,-4.7) (-0.5,-4.7) (-1.0,-4.5) (-2,-4) (-3,-3) (-5,-0.8) (-5,0.8) (-3,3.5) (-2,4.5) (-1,5) (-0.3,5.2) (0.3,5.2) (1,5) (2,4) (3.1,0.7) 
(2.95,0.2)};
\filldraw[fill=black!20!white, draw=blue!50!black] (0,0) ellipse (2 and 1) node {$C_1$};
\filldraw[fill=black!20!white, draw=blue!50!black] (-4,0) ellipse (0.5 and 0.75);
\filldraw[fill=black!20!white, draw=blue!50!black] (4,0) ellipse (0.5 and 0.75) node {$C_2$}; 
\filldraw[fill=black!20!white, draw=blue!50!black] (0,4) ellipse (1.5 and 1);
\filldraw[fill=black!20!white, draw=blue!50!black] (0,-3.5) ellipse (1.5 and 1);
\filldraw[fill=black!20!white, draw=blue!50!black] (0,1.6) ellipse (0.5 and 0.2);
\filldraw[fill=black!20!white, draw=blue!50!black] (0,2.4) ellipse (0.5 and 0.2);
\end{scope}
\end{tikzpicture}

%% file: pic-link-13.tex
\begin{tikzpicture}[scale=1.5]

     \draw[->,>=latex,thick, gray] (-1.6,0)--(1.6,0); 
     \draw[->,>=latex,thick, gray] (0,-1.6)--(0,1.6); 

\begin{scope}[rotate=-20]

  \def\myc{-0.5}
   \draw [thick, color=gray,samples=100,smooth, domain=-1:1]  plot ({\myc - \x*\x}, {\x} );
  \draw[red,thick] (\myc,-1.3)--(\myc,1.3) node[pos=0, below] {$r<r_0$}  ;

  \def\myc{0}
   \draw [thick, color=gray,samples=100,smooth, domain=-1:1]  plot ({\myc - \x*\x}, {\x} );    
  \draw[red,thick] (\myc,-1.3)--(\myc,1.3) node[pos=0, below] {$r=r_0$} ;

  \def\myc{0.5}
   \draw [thick, color=gray,samples=100,smooth, domain=-1:1]  plot ({\myc - \x*\x}, {\x} );
  \draw[red,thick] (\myc,-1.3)--(\myc,1.3) node[pos=0, below] {$r>r_0$}  ;

    \fill (0,0) circle (1pt) node[below right] {$0$};
    \fill (1,0) circle (1pt) node[right] {$\ol\omega$};
    \fill (40:1) circle (1pt) node[right] {$\omega$};
    \draw (0,0)--(1,0);     
\end{scope}

\end{tikzpicture}

%% file: pic-link-14.tex
\begin{tikzpicture}[scale=1.2, font=\small]

     \draw[->,>=latex,thick, gray] (-5,0)--(1.6,0); 
     \draw[->,>=latex,thick, gray] (0,-1.6)--(0,1.6); 

\begin{scope}[rotate=0]

  \def\myc{-0.5}
 \draw [very thick, color=gray,samples=100,smooth, domain=-1:1] (-2,0) ellipse (2cm and 1cm);
 \draw [very thick, color=blue,samples=100,smooth, domain=-1:1] (0,0) arc (0:-30:2cm and 1cm);
 \draw [very thick, color=blue,samples=100,smooth, domain=-1:1] (0,0) arc (0:30:2cm and 1cm);

  \node[below] at (-1,-0.1) {$r<r_0$} ;
  \draw[red,thick] (\myc,-1.3)--(\myc,1.3)  ;

  \def\myc{0}
 \draw [very thick, color=gray,samples=100,smooth, domain=-1:1] (-2,0) ellipse (1.5cm and 0.75cm);
 \draw [very thick, color=blue,samples=100,smooth, domain=-1:1] (-0.5,0) arc (0:-30:1.5cm and 0.75cm);
 \draw [very thick, color=blue,samples=100,smooth, domain=-1:1] (-0.5,0) arc (0:30:1.5cm and 0.75cm);

   \node[below=7pt]  at (-0.25,-0.45) {$r=r_0$} ;
  \draw[red,thick] (\myc,-1.3)--(\myc,1.3);

  \def\myc{0.5}
 \draw [very thick, color=gray,samples=100,smooth, domain=-1:1] (-2,0) ellipse (2.5cm and 1.25cm);
 \draw [very thick, color=blue,samples=100,smooth, domain=-1:1] (0.5,0) arc (0:-30:2.5cm and 1.25cm);
 \draw [very thick, color=blue,samples=100,smooth, domain=-1:1] (0.5,0) arc (0:30:2.5cm and 1.25cm); 
   \node[below] at (0.5,-0.5) {$r>r_0$} ;
  \draw[red,thick] (\myc,-1.3)--(\myc,1.3) ;

    \fill (0,0) circle (1pt) node[below right] {$0$};
    \fill (1,0) circle (1pt) node[below] {$\ol\omega$};
    \draw (0,0)--(1,0);     
\end{scope}

\end{tikzpicture}

%% file: pic-link-15.tex
\begin{tikzpicture}[scale=1.2, font=\small]

     \draw[->,>=latex,thick, gray] (-1.6,0)--(2.6,0); 
     \draw[->,>=latex,thick, gray] (0,-1.6)--(0,1.6); 

\begin{scope}[rotate=0]

  \def\myc{-0.5}
  \draw [very thick, color=blue,samples=100,smooth, domain=-1:1]  plot ({\myc - 0.5*\x*\x}, {\x} );
  \draw [very thick, color=gray,samples=100,smooth, domain=-1:1]  plot ({1.5-\myc + 0.5*\x*\x}, {\x} );
  \node[below] at (-1,-1) {$r>r_0$} ;
  \draw[red,thick] (\myc,-1.3)--(\myc,1.3)  ;

  \def\myc{0}
  \draw [thick, color=blue,samples=100,smooth, domain=-1:1]  plot ({\myc - 0.5*\x*\x}, {\x} );
  \draw [very thick, color=gray,samples=100,smooth, domain=-1:1]  plot ({1.5-\myc + 0.5*\x*\x}, {\x} );  
   \node[below]  at (-0.25,-1.3) {$r=r_0$} ;
  \draw[red,thick] (\myc,-1.3)--(\myc,1.3);

  \def\myc{0.5}
   \draw [thick, color=blue,samples=100,smooth, domain=-1:1]  plot ({\myc -0.5* \x*\x}, {\x} );
  \draw [very thick, color=gray,samples=100,smooth, domain=-1:1]  plot ({1.5-\myc + 0.5*\x*\x}, {\x} );
   \node[below] at (0.25,-1) {$r<r_0$} ;
  \draw[red,thick] (\myc,-1.3)--(\myc,1.3) ;

    \fill (0,0) circle (1pt) node[below right] {$0$};
    \fill (-1,0) circle (1pt) node[below] {$\ol\omega$};
    \draw (0,0)--(1,0);     
\end{scope}

\end{tikzpicture}

%% file: pic-link-19.tex
\begin{tikzpicture}[scale=1]

\begin{scope}[xshift=-4cm]

     \draw[->,>=latex,thick, gray] (-1.5,0)--(1,0); 
     \draw[->,>=latex,thick, gray] (0,-1.2)--(0,1.3); 
    \node at (0.5,-1) {$\kappa<0$};
  \def\myc{-0.5}
   \draw [very thick, color=blue,samples=100,smooth, domain=-1:1]  plot ({\myc - \x*\x}, {\x} );

\draw[->,>=latex,thick] (-0.5,-1.5)--(-0.5,-3.5) node[midway,left]{$t\mapsto t^ 2$};

\begin{scope}[yshift=-5cm]
     \draw[->,>=latex,thick, gray] (-1.5,0)--(1,0); 
     \draw[->,>=latex,thick, gray] (0,-1.2)--(0,1.3); 
   \draw [very thick, color=blue,samples=100,smooth, domain=-1:1]  plot ({\myc*\myc - \x*\x}, {\myc*\x-\x*\x*\x} );
\end{scope}

\end{scope}

\begin{scope}[xshift=0cm]

     \draw[->,>=latex,thick, gray] (-1.5,0)--(1,0); 
     \draw[->,>=latex,thick, gray] (0,-1.2)--(0,1.3); 
    \node at (0.5,-1) {$\kappa=0$};
  \def\myc{0}
   \draw [very thick, color=blue,samples=100,smooth, domain=-1:1]  plot ({\myc - \x*\x}, {\x} );

\draw[->,>=latex,thick] (-0.5,-1.5)--(-0.5,-3.5) node[midway,left]{$t\mapsto t^ 2$};

\begin{scope}[yshift=-5cm]
     \draw[->,>=latex,thick, gray] (-1.5,0)--(1,0); 
     \draw[->,>=latex,thick, gray] (0,-1.2)--(0,1.3); 
   \draw [very thick, color=blue,samples=100,smooth, domain=-1.05:1.05]  plot ({\myc*\myc - \x*\x}, {\myc*\x-\x*\x*\x} );
\end{scope}

\end{scope}

\begin{scope}[xshift=4cm]

     \draw[->,>=latex,thick, gray] (-1.5,0)--(1,0); 
     \draw[->,>=latex,thick, gray] (0,-1.2)--(0,1.3); 
    \node at (0.5,-1) {$\kappa>0$};
  \def\myc{0.5}
   \draw [very thick, color=blue,samples=100,smooth, domain=-1:1]  plot ({\myc - \x*\x}, {\x} );

\draw[->,>=latex,thick] (-0.5,-1.5)--(-0.5,-3.5) node[midway,left]{$t\mapsto t^ 2$};

\begin{scope}[yshift=-5cm]
     \draw[->,>=latex,thick, gray] (-1.5,0)--(1,0); 
     \draw[->,>=latex,thick, gray] (0,-1.2)--(0,1.3); 
   \draw [very thick, color=blue,samples=100,smooth, domain=-1.15:1.15]  plot ({\myc*\myc - \x*\x}, {\myc*\x-\x*\x*\x} );
\end{scope}

\end{scope}

\end{tikzpicture}

%% file: pic-link-20.tex
\begin{tikzpicture}[scale=2]

     \draw[->,>=latex,thick, gray] (-1.6,0)--(1.6,0); 
     \draw[->,>=latex,thick, gray] (0,-1.6)--(0,1.6); 

\begin{scope}[rotate=-20]

%
  \def\myc{0}
   \draw[red,thick] (\myc,-1.3)--(\myc,1.3) node[pos=0, below] {$r=r_0$} ;

  \def\myc{0.5}
  \draw [thick, color=blue,samples=100,smooth, domain=-1.1:1.1]  plot ({\myc - \x*\x}, {\x} ) ;

  \draw[red,thick] (\myc,-1.3)--(\myc,1.3) node[pos=0, below] {$r>r_0$};

  \node[green!60!black] at (-0.9,1.1) {$t \to {+\infty}$};
  \node[green!60!black] at (-0.9,-1.1) {$t \to {-\infty}$};
  \fill[green!60!black] (0,0.71) circle (1pt) node[right] {$t_{+}$};
  \fill[green!60!black] (0,-0.71) circle (1pt) node[below right] {$t_{-}$};
  \fill[green!60!black] (\myc,0) circle (1pt) node[right] {$t_{0}$};

    \fill (0,0) circle (1pt) node[below right] {$0$};
    \fill (1,0) circle (1pt) node[right] {$\ol\omega$};
    \fill (40:1) circle (1pt) node[right] {$\omega$};
    \draw (0,0)--(1,0);

  \node at (-1.5,0.5) {Case $\mathcal{E}$};
\end{scope}

\begin{scope}[xshift=4cm]
     \draw[->,>=latex,thick, gray] (-1.6,0)--(1.6,0); 
     \draw[->,>=latex,thick, gray] (0,-1.6)--(0,1.6); 

\begin{scope}[rotate=-20]

%
  \def\myc{0}
   \draw[red,thick] (\myc,-1.3)--(\myc,1.3) node[pos=0, below] {$r=r_0$} ;

  \def\myc{0.5}
  \draw [thick, color=blue,samples=100,smooth, domain=-1.1:1.1]  plot ({\myc - \x*\x}, {\x} ) ;

  \draw[red,thick] (\myc,-1.3)--(\myc,1.3) node[pos=0, below] {$r<r_0$}  ;

  \node[green!60!black] at (-0.9,1.1) {$t \to {+\infty}$};
  \node[green!60!black] at (-0.9,-1.1) {$t \to {-\infty}$};
  \fill[green!60!black] (0,0.71) circle (1pt) node[right] {$t_{+}$};
  \fill[green!60!black] (0,-0.71) circle (1pt) node[below right] {$t_{-}$};
  \fill[green!60!black] (\myc,0) circle (1pt) node[right] {$t_{0}$};

    \fill (0,0) circle (1pt) node[below right] {$0$};
    \fill (-1,0) circle (1pt) node[left] {$\ol\omega$};
    \fill (40:-1) circle (1pt) node[left] {$\omega$};

    \draw (0,0)--(-1,0);    
   \node at (-1.5,0.5) {Case $\mathcal{H}$};
\end{scope}
\end{scope}

\end{tikzpicture}

%% file: pic-link-21.tex
\begin{tikzpicture}[scale=1.2]

\begin{scope}[xshift=0cm]

     \draw[->,>=latex,thick, gray] (-1.6,0)--(1.6,0); 
     \draw[->,>=latex,thick, gray] (0,-1.6)--(0,1.6); 

  \def\myc{0.5}
  \draw [very thick, color=blue,samples=100,smooth, domain=-1.1:1.1]  plot ({\myc - \x*\x}, {\x} ) ;

  \node[green!60!black] at (-0.9,1.3) {$t \to {+\infty}$};
  \node[green!60!black] at (-0.9,-1.3) {$t \to {-\infty}$};
  \fill[green!60!black] (0,0.71) circle (1pt) node[above right] {$t_{+}$};
  \fill[green!60!black] (0,-0.71) circle (1pt) node[below right] {$t_{-}$};
  \fill[green!60!black] (\myc,0) circle (1pt) node[below right] {$t_{0}$};

    \fill (0,0) circle (1pt) node[below right] {$0$};

  \node at (-1.5,0.5) {Case $\mathcal{E}$};

\draw[->,>=latex,thick] (-0.5,-1.5)--(-0.5,-3.5) node[midway,left]{$t\mapsto t^ 2$};

\begin{scope}[yshift=-5cm]
     \draw[->,>=latex,thick, gray] (-1.5,0)--(1.6,0); 
     \draw[->,>=latex,thick, gray] (0,-1.6)--(0,1.6); 
  \def\myc{0.7}

   \draw [very thick, color=blue,samples=100,smooth, domain=-1.2:-0.9]  plot ({\myc*\myc - \x*\x}, {\myc*\x-\x*\x*\x} );
    \draw [very thick, color=blue,samples=100,smooth, domain=-0.75:1.2]  plot ({\myc*\myc - \x*\x}, {\myc*\x-\x*\x*\x} );

  \node[green!60!black] at (-0.9,1.1) {$(t \to {-\infty})^2$};
  \node[green!60!black] at (-0.9,-1.1) {$(t \to {+\infty})^2$};
  \fill[green!60!black] (-0.21,0) circle (1.5pt) node (wr1) {};
  \node[red, below=5pt] at (-0.21,0) {$\boxplus$};
  \node[green!60!black] (wr2) at (1,1) {$t_+^2=t_-^2$};
  \draw[dashed, green!20!black,->] (wr2) .. controls +(left:5mm) and +(up:5mm).. (wr1);
  \fill[green!60!black] (\myc*\myc,0) circle (1pt) node[below right] {$t_{0}^2$};

\end{scope}

\end{scope}

\begin{scope}[xshift=5cm]

     \draw[->,>=latex,thick, gray] (-1.6,0)--(1.6,0); 
     \draw[->,>=latex,thick, gray] (0,-1.6)--(0,1.6); 

  \def\myc{0.5}
  \draw [very thick, color=blue,samples=100,smooth, domain=-1.1:1.1]  plot ({\myc - \x*\x}, {\x} ) ;

  \node[green!60!black] at (-0.9,1.3) {$t \to {+\infty}$};
  \node[green!60!black] at (-0.9,-1.3) {$t \to {-\infty}$};
  \fill[green!60!black] (0,0.71) circle (1pt) node[above right] {$t_{+}$};
  \fill[green!60!black] (0,-0.71) circle (1pt) node[below right] {$t_{-}$};
  \fill[green!60!black] (\myc,0) circle (1pt) node[below right] {$t_{0}$};

    \fill (0,0) circle (1pt) node[below right] {$0$};

  \node at (-1.5,0.5) {Case $\mathcal{H}$};

\draw[->,>=latex,thick] (-0.5,-1.5)--(-0.5,-3.5) node[midway,left]{$t\mapsto t^ 2$};

\begin{scope}[yshift=-5cm]
     \draw[->,>=latex,thick, gray] (-1.5,0)--(1.6,0); 
     \draw[->,>=latex,thick, gray] (0,-1.6)--(0,1.6); 
  \def\myc{0.7}

   \draw [very thick, color=blue,samples=100,smooth, domain=-1.2:+0.75]  plot ({\myc*\myc - \x*\x}, {\myc*\x-\x*\x*\x} );
    \draw [very thick, color=blue,samples=100,smooth, domain=0.9:1.2]  plot ({\myc*\myc - \x*\x}, {\myc*\x-\x*\x*\x} );

  \node[green!60!black] at (-0.9,1.1) {$(t \to {-\infty})^2$};
  \node[green!60!black] at (-0.9,-1.1) {$(t \to {+\infty})^2$};
  \fill[green!60!black] (-0.21,0) circle (1.5pt) node (wrr1) {};
  \node[green!60!black] (wrr2) at (1,1) {$t_+^2=t_-^2$};
  \draw[dashed, green!20!black,->] (wrr2) .. controls +(left:5mm) and +(up:5mm).. (wrr1);
\node[red, below=5pt] at (-0.21,0){$\boxminus$};
  \fill[green!60!black] (\myc*\myc,0) circle (1pt) node[below right] {$t_{0}^2$};

\end{scope}

\end{scope}

\end{tikzpicture}

%% file: pic-link-22.tex
\begin{tikzpicture}[scale=2]

     \draw[->,>=latex,thick, gray] (-1.6,0)--(1.6,0); 
     \draw[->,>=latex,thick, gray] (0,-1.6)--(0,1.6); 

\begin{scope}[rotate=-20]

  \def\myc{0.5}
  \draw [thick, color=blue,samples=100,smooth, domain=-1.1:1.1]  plot ({\myc - \x*\x}, {\x} ) ;

  \draw[red,thick] (\myc,-1.3)--(\myc,1.3) node[pos=0, below] {$r>r_0$};

  \fill[green!60!black] (\myc,0) circle (1pt) node[right] {$t_{0}$};

    \fill (0,0) circle (1pt) node[below right] {$0$};
    \fill (1,0) circle (1pt) node[right] {$\ol\omega$};
    \draw (0,0)--(1,0);

  \node at (-1.5,0.5) {Case $\mathcal{E}$};
\end{scope}

\begin{scope}[xshift=4cm]
     \draw[->,>=latex,thick, gray] (-1.6,0)--(1.6,0); 
     \draw[->,>=latex,thick, gray] (0,-1.6)--(0,1.6); 

\begin{scope}[rotate=-20]

  \def\myc{0.5}
  \draw [thick, color=blue,samples=100,smooth, domain=-1.1:1.1]  plot ({\myc - \x*\x}, {\x} );

  \draw[red,thick] (\myc,-1.3)--(\myc,1.3) node[pos=0, below] {$r<r_0$}  ;

  \fill[green!60!black] (\myc,0) circle (1pt) node[right] {$t_{0}$};

    \fill (0,0) circle (1pt) node[below right] {$0$};
    \fill (-1,0) circle (1pt) node[left] {$\ol\omega$};

    \draw (0.5,0)--(-1,0);    
   \node at (-1.5,0.5) {Case $\mathcal{H}$};
\end{scope}
\end{scope}

\end{tikzpicture}

%% file: pic-link-28.tex
\begin{tikzpicture}[scale=2.5]

     \draw[->,>=latex,thick, gray] (-1,0)--(1.5,0); 
     \draw[->,>=latex,thick, gray] (0,-0.7)--(0,0.7); 
  \def\myc{0.7}

   \draw [very thick, color=blue,samples=100,smooth, domain=-1.1:1.1]  plot ({\myc*\myc - \x*\x}, {\myc*\x-\x*\x*\x} );

  \fill[green!60!black] (\myc*\myc-0.2,0) circle (1pt) node(mycircle) {};
  \node[green!60!black] (mycircle2) at (1,-0.5) {$\pi_r(\epsilon-\epsilon\eta)$};
  \draw[dashed,green!20!black,->] (mycircle2) ..controls +(left:5mm) .. (mycircle);
  \fill[green!60!black] (\myc*\myc,0) circle (1pt) node[above right] {$z_0=\pi_r(\epsilon)$};
  \draw [->,>=latex,thick, green!60!black] (\myc*\myc,0)--+(0.5,0) node[below]{$w$};

\end{tikzpicture}

%% file: pic-link-23.tex
\begin{tikzpicture}[scale=1.3]

\begin{scope}[xshift=0cm]
   \node at (-1.5,0.5) {Case $\mathcal{E}$};
     \draw[->,>=latex,thick, gray] (-1.4,0)--(2,0); 
     \draw[->,>=latex,thick, gray] (0,-1.4)--(0,1.4); 
  \def\myc{0.7}

   \draw [very thick, color=blue,samples=100,smooth, domain=-1.2:-0.9]  plot ({\myc*\myc - \x*\x}, {\myc*\x-\x*\x*\x} );
    \draw [very thick, color=blue,samples=100,smooth, domain=-0.75:1.2]  plot ({\myc*\myc - \x*\x}, {\myc*\x-\x*\x*\x} );

  \fill[green!60!black] (\myc*\myc,0) circle (1pt); 
  \draw [->,>=latex,thick, green!60!black] (\myc*\myc,0)--+(0,1)
node[pos=0, below right, red]{$\oplus$} node[left]{$v$};
  \draw [->,>=latex,thick, green!60!black] (\myc*\myc,0)--+(1,0) node[above]{$w$};
\end{scope}

\begin{scope}[xshift=5cm]
   \node at (-1.5,0.5) {Case $\mathcal{H}$};

     \draw[->,>=latex,thick, gray] (-1.4,0)--(2,0); 
     \draw[->,>=latex,thick, gray] (0,-1.4)--(0,1.4); 
  \def\myc{0.7}

   \draw [very thick, color=blue,samples=100,smooth, domain=-1.2:+0.75]  plot ({\myc*\myc - \x*\x}, {\myc*\x-\x*\x*\x} );
    \draw [very thick, color=blue,samples=100,smooth, domain=0.9:1.2]  plot ({\myc*\myc - \x*\x}, {\myc*\x-\x*\x*\x} );

  \fill[green!60!black] (\myc*\myc,0) circle (1pt); 
  \draw [->,>=latex,thick, green!60!black] (\myc*\myc,0)--+(0,-1)
node[pos=0, below right, red]{$\ominus$} node[left]{$v$};
  \draw [->,>=latex,thick, green!60!black] (\myc*\myc,0)--+(1,0) node[above]{$w$};
\end{scope}

\end{tikzpicture}